\documentclass[12pt]{amsart}

\usepackage{framed}
\usepackage{braket}
\usepackage{amsmath,amssymb,amsthm}
\usepackage{color}
\usepackage{bm}
\usepackage[all]{xy}
\usepackage{amscd}
\usepackage{graphicx}
\usepackage{ascmac}
\usepackage{BOONDOX-frak, mathrsfs}
\usepackage[top=30truemm,bottom=30truemm,left=25truemm,right=25truemm]{geometry}
\usepackage{mathtools}
\mathtoolsset{showonlyrefs=true}
\usepackage{tikz}
\usetikzlibrary{cd}

\makeatletter
\@addtoreset{equation}{section}

\makeatother

\makeatletter
\@namedef{subjclassname@2020}{\textup{2020} Mathematics Subject Classification}
\makeatother

\newcommand{\Z}{\mathbb{Z}}

\newcommand{\Q}{\mathbb{Q}}

\newcommand{\cA}{\mathcal{A}}

\newcommand{\cL}{\mathcal{L}}

\newcommand{\cP}{\mathcal{P}}

\newcommand{\wtil}[1]{\widetilde{#1}}
\newcommand{\ol}[1]{\overline{#1}}

\DeclareMathOperator{\Gal}{Gal}
\DeclareMathOperator{\Coker}{Cok}
\DeclareMathOperator{\Cok}{Cok}
\DeclareMathOperator{\Ker}{Ker}

\DeclareMathOperator{\Fitt}{Fitt}

\DeclareMathOperator{\pd}{pd}

\DeclareMathOperator{\id}{id}
\DeclareMathOperator{\Cl}{Cl}
\DeclareMathOperator{\ord}{ord}
\DeclareMathOperator{\Hom}{Hom}
\DeclareMathOperator{\Ext}{Ext}

\DeclareMathOperator{\rank}{rank}

\DeclareMathOperator{\cha}{char}

\let\oldenumerate\enumerate
\renewcommand{\enumerate}{
   \oldenumerate
   \setlength{\itemsep}{1pt}
   \setlength{\parskip}{0pt}
   \setlength{\parsep}{0pt}
}
\let\olditemize\itemize
\renewcommand{\itemize}{
   \olditemize
   \setlength{\itemsep}{1pt}
   \setlength{\parskip}{0pt}
   \setlength{\parsep}{0pt}
}

\theoremstyle{plain}
\newtheorem{thm}{Theorem}[section]
\newtheorem{lem}[thm]{Lemma}

\newtheorem{prop}[thm]{Proposition}
\newtheorem{cor}[thm]{Corollary}
\newtheorem{claim}[thm]{Claim}
\newtheorem{ass}[thm]{Assumption}

\theoremstyle{definition}
\newtheorem{defn}[thm]{Definition}

\newtheorem{rem}[thm]{Remark}
\newtheorem{eg}[thm]{Example}

\newcommand{\bE}{\mathbb{E}}
\newcommand{\ul}[1]{\underline{#1}}
\newcommand{\hZ}{\hat{\Z}}
\newcommand{\GamLambda}{{}_{\Gamma}\Lambda}
\DeclareMathOperator{\Div}{Div}
\DeclareMathOperator{\Pic}{Pic}
\DeclareMathOperator{\Jac}{Jac}

\DeclareMathOperator{\Frac}{Frac}
\DeclareMathOperator{\depth}{depth}
\DeclareMathOperator{\Image}{Im}
\DeclareMathOperator{\Tor}{Tor}

\title
[Jacobian groups of graphs]
{Fitting ideals of Jacobian groups of graphs}
\author[T. Kataoka]{Takenori Kataoka}
\address{Department of Mathematics, Faculty of Science Division II, Tokyo University of Science.
1-3 Kagurazaka, Shinjuku-ku, Tokyo 162-8601, Japan}
\email{tkataoka@rs.tus.ac.jp}
\date{\today}

\begin{document}

\maketitle

\begin{abstract}
The Jacobian group of a graph is a finite abelian group through which we can study the graph in an algebraic way.
When the graph is a finite abelian covering of another graph, the Jacobian group is equipped with the action of the Galois group.
In this paper, we study the Fitting ideal of the Jacobian group as a module over the group ring.
We also study the corresponding question for infinite coverings.
Additionally, this paper includes module-theoretic approach to Iwasawa theory for graphs.
\end{abstract}

\section{Introduction}\label{sec:intro}

For a (finite) graph $X$, let $\Jac(X)$ denote its Jacobian group, which is also known as the sandpile group.
It is known that $\Jac(X)$ is a finite abelian group.
See \S \ref{sec:defn_J_graph} for basic notions about graphs and their Jacobian groups.

Let $Y/X$ be an abelian covering of connected graphs and $\Gamma$ its Galois group.
Then the Jacobian group $\Jac(Y)$ is naturally equipped with a $\Z[\Gamma]$-module structure, which is the main theme of the present paper.
More concretely, we focus on its Fitting ideal (see \S \ref{ss:Fitt_defn} for the definition).
The main result of this paper gives a complete description of the Fitting ideal
\[
\Fitt_{\Z[\Gamma]/(N_{\Gamma})}(\Jac(Y)/N_{\Gamma}\Jac(Y)).
\]
Here, $N_{\Gamma}$ denotes the norm element of $\Gamma$.
See \S \ref{ss:main_result} for the concrete statement.
This result can be regarded as an analogue of a result of Atsuta and the author \cite{AK21} in number theory, as explained in Remark \ref{rem:AK} below.

We will also study infinite abelian coverings in a sense in \S \ref{sec:Fitt_infin}.
In this case, we obtain a complete description of the Fitting ideal of the associated Jacobian group itself, without dividing by the Jacobian group of the base graph.

The study of infinite coverings of graphs is morally an analogue of Iwasawa theory in number theory.
Recently such an analogue of Iwasawa theory for graphs is developed (e.g., Gonet \cite{Gon22} and a series of work of Valli\`{e}res including \cite{MV21b}).
The present work is directly inspired by this situation surrounding graph theory and Iwasawa theory.
Moreover, in \S \ref{sec:Iwa}, as a digression, we give concise proofs of the analogues of the Iwasawa class number formula and of Kida's formula.
Those results are already proved by others, but the author hopes that this makes Iwasawa theory for graphs more accessible for those who are familiar with techniques in Iwasawa theory.

\subsection{Main result for finite coverings}\label{ss:main_result}

Now we state the main result.
Let $Y/X$ be an abelian covering of connected graphs and $\Gamma$ its Galois group.
In \S \ref{sec:defn_J_graph}, we will introduce an element
$
Z_{Y/X} \in \Z[\Gamma],
$
which is related to the (equivariant) Ihara zeta function via the three-term determinant formula (see \S \ref{ss:Ihara}).

Let us take a decomposition
\[
\Gamma = \Delta_1 \times \cdots \times \Delta_s,
\]
where $\Delta_l$ is a cyclic group of order $n_l \geq 1$ for each $1 \leq l \leq s$.
For each $l$, we fix a generator $\sigma_l$ of $\Delta_l$ and define $\tau_l$, $\nu_l \in \Z[\Delta_l] \subset \Z[\Gamma]$ by
\[
\tau_l = \sigma_l - 1,
\quad
\nu_l = 1 + \sigma_l + \sigma_l^2 + \cdots + \sigma_l^{n_l-1}.
\]
Then $\nu_l$ is the norm element of $\Delta_l$.
The norm element of $\Gamma$, defined by $N_{\Gamma} = \sum_{\gamma \in \Gamma} \gamma$, satisfies $N_{\Gamma} = \nu_1 \cdots \nu_s$.
We also define an element $D_l \in \Z[\Delta_l]$ by
\[
D_l = \sigma_l + 2 \sigma_l^2 + \cdots + (n_l-1) \sigma_l^{n_l-1}.
\]

The main result is the following.
The proof will be given in \S \S \ref{sec:Fin_arith}--\ref{sec:fin_alg}.

\begin{thm}\label{thm:main_fin}
We have
\begin{align}
& \Fitt_{\Z[\Gamma]/(N_{\Gamma})}(\Jac(Y)/N_{\Gamma}\Jac(Y))\\
& \quad = Z_{X, \Gamma} \sum_{l=1}^s \left( \nu_1 \cdots \nu_{l-1} \cdot \frac{D_l}{n_l} \cdot \nu_{l+1} \cdots \nu_s \right) \\
& \quad \quad + Z_{X, \Gamma} \left(\nu_1^{e_1} \cdots \nu_s^{e_s} \tau_1^{f_1} \cdots \tau_s^{f_s} \, \middle| \, \begin{matrix} 0 \leq e_l \leq 1, f_l \geq 0, \\ e_1 + \cdots + e_s + f_1 + \cdots + f_s = s- 2 \end{matrix} \right).
\end{align}
Here the right hand side, which is literally an ideal of $\Z[\Gamma]$, is regarded as an ideal of $\Z[\Gamma]/(N_{\Gamma})$ via the natural projection.
\end{thm}

\begin{rem}\label{rem:Kol}
In the theory of Euler systems in number theory, the element $D_l$ is often called the Kolyvagin derivative operator and plays an important role.
It seems to be interesting that this element is useful in computation that is not directly related to the theory of Euler systems.
Indeed, as far as the author is aware, it is the first appearance of the Kolyvagin derivative operator in computation of Fitting ideals.
\end{rem}

In Proposition \ref{prop:NormY}, we will show an isomorphism
$
\Jac(X) \simeq N_{\Gamma} \Jac(Y).
$
Therefore, Theorem \ref{thm:main_fin} can be viewed as a description of the Fitting ideal of $\Jac(Y)/\Jac(X)$.
Unfortunately, determining $\Fitt_{\Z[\Gamma]}(\Jac(Y))$ itself seems to be out of our reach.

\begin{rem}\label{rem:AK}
Theorem \ref{thm:main_fin} can be regarded as an analogue of a result of Atsuta and the author \cite{AK21} on the Fitting ideal of class groups in number theory.
Let us roughly review the result.
Let $L/K$ be a finite abelian extension of number fields such that $K$ is totally real and $L$ is a CM-field.
We consider the $T$-ray class group $\Cl_L^T$ of $L$, where $T$ is an auxiliary finite set of finite primes of $K$.
Then in \cite{AK21}, (assuming the relevant equivariant Tamagawa number conjecture), we gave a complete description of the Fitting ideal
\[
\Fitt_{\Z[\Gal(L/K)]^-}(\Cl_L^{T, -}),
\]
where the superscript $(-)^-$ denotes the minus component with respect to the complex conjugation.
The plus component seems to be out of our reach.
This is analogous to the situation in Theorem \ref{thm:main_fin}.
\end{rem}

\subsection{Idea of proof}\label{ss:comments}

In number theory, there are already a number of studies on the Fitting ideals of arithmetic modules.
In particular, in the field of Iwasawa theory, the Fitting ideals of Iwasawa modules have been studied by Greither--Kurihara (e.g., \cite{GK15}, \cite{GKT20} with Tokio, etc.).

In \cite{Kata_05}, the author introduced the ``shift theory'' of Fitting ideals, which we will review in \S \ref{ss:shift_defn}.
It is a useful technique to compute the Fitting ideals of arithmetic modules.
Roughly speaking, the technique has two steps: 
\begin{itemize}
\item[(1)]
we study the arithmetic module and obtain a description of the Fitting ideal using an explicit algebraic factor, and 
\item[(2)]
we compute the algebraic factor in a purely algebraic way.
\end{itemize}
The first applications of this theory in \cite{Kata_05} were to Iwasawa theory, i.e., infinite extensions of number fields.
Then Atsuta and the author \cite{AK21} made an application to class groups of number fields as explained in Remark \ref{rem:AK}.

In this paper, we again employ the shift theory as the basic technique to prove both Theorem \ref{thm:main_fin} and the result for infinite coverings.
Therefore, we follow the two steps explained above.
See Theorems \ref{thm:Fitt_fin1} and \ref{thm:Fitt_fin} for the first and second steps, respectively.

The algebraic factor to be computed in Theorem \ref{thm:Fitt_fin} seems to be a new one as discussed in Remark \ref{rem:Kol}.
Moreover, the proof of the formula is not easy, and it is the most technical part in this paper.
We employ a combinatoric method (using graph theory), though the problem is purely algebraic.
Note that Greither--Kurihara--Tokio \cite{GKT20} also applied graph theory to compute Fitting ideals.

On the other hand, the algebraic factor to be computed for the infinite covering case does not seem to be very new.
Indeed, it can be computed by using a formula that has been obtained in \cite{AK21}.
However, a relatively minor issue is that in this paper we deal with commutative rings that are not necessarily noetherian.
Therefore, we have to generalize the shift theory, which was developed only over noetherian rings in \cite{Kata_05}.
This is explained in \S \ref{ss:shift_defn}.

\subsection{Organization of this paper}\label{ss:org}

In \S \ref{sec:defn_J_graph}, we introduce basic notions about graphs.
Then in \S \ref{sec:Fin_arith}, we reduce the proof of Theorem \ref{thm:main_fin} to an algebraic problem.
The algebraic problem is solved in \S \ref{sec:fin_alg}.

After preliminaries in \S \ref{sec:pre} on voltage graphs and their derived graphs, in \S \ref{sec:Fitt_infin} we state and prove the result for infinite coverings.

In \S \ref{sec:dual}, we observe the self-duality of the Jacobian groups, which is a contrast to the analogue in number theory.
In \S \ref{sec:Iwa}, we explain the module-theoretic approach to Iwasawa theory for graphs.
Both \S \S \ref{sec:dual} and \ref{sec:Iwa} can be read independently.
Finally in \S \ref{sec:Fitt_gen}, we review the definition of the Fitting ideals and the shift theory.

\section{Graphs and their Jacobian groups}\label{sec:defn_J_graph}

In this section, we introduce basic notions about graphs and their Jacobian groups.
References are Corry--Perkinson \cite[Chapters 1 and 2]{CP18} and Baker--Norine \cite[\S 1.3]{BN07}.
One can also consult recent papers of Valli\`eres such as Hammer--Mattman--Sands--Valli\`eres \cite[\S 2]{HMSV} and Ray--Valli\`eres \cite[\S 2]{RV22}.

In this paper, graphs are always assumed to be finite.
We allow graphs to have multi-edges and loops (so one may call them multigraphs).
More precisely, we define graphs as follows, using Serre's formalism \cite[Chapter I, \S 2.1]{Ser80}.

\begin{defn}
A graph $X$ consists of a finite set $V_X$ of vertices, a finite set $\bE_X$ of edges, an automorphism of $\bE_X$ denoted by $e \mapsto \ol{e}$, and two maps $s_X$, $t_X$ from $\bE_X$ to $V_X$ (often abbreviated to $s = s_X$ and $t = t_X$), satisfying the following:
\begin{itemize}
\item
For any $e \in \bE_X$, we have $\ol{e} \neq e$ and $\ol{\ol{e}} = e$, that is, the automorphism $e \mapsto \ol{e}$ is an involution of $\bE_X$ without fixed points.
\item
For any $e \in \bE_X$, we have $s(\ol{e}) = t(e)$ and $t(\ol{e}) = s(e)$.
\end{itemize}

Each element $e \in \bE_X$ is regarded as an edge from $s(e)$ to $t(e)$, and $\ol{e}$ is regarded as the inverse of $e$.
Let us write $E_X$ for the quotient set of $\bE_X$ obtained by identifying $e$ and $\ol{e}$.
Then we have a canonical two-to-one map from $\bE_X$ to $E_X$, so their cardinalities satisfy $\# E_X = \frac{1}{2} \cdot \# \bE_X$.

For each $v \in V_X$, we define
\[
\bE_{X, v} = \{ e \in \bE_X \mid s(e) = v\},
\]
which is the set of edges that start from $v$.
\end{defn}

This formalism involving $\bE_X$ and the involution will be useful for introducing voltage graphs in \S \ref{sec:pre}.
On the other hand, we can regard $X$ as an undirected (multi-)graph through the quotient set $E_X$.
The notions that we will introduce in this section are essentially defined for the associated undirected graph structure.

In this paper, for simplicity, we usually deal with connected graphs.
Being connected is not an essential assumption since the general case can be easily deduced.

In the rest of this section, let $X$ be a connected graph.

\begin{defn}
We define the divisor group $\Div(X)$ as the free $\Z$-module on the set $V_X$, namely, 
\[
\Div(X) = \bigoplus_{v \in V_X} \Z [v].
\]
Let us write
\[
\deg_X: \Div(X) \to \Z
\]
for the $\Z$-homomorphism that sends $[v]$ to $1$ for any $v \in V_X$.
We define $\Div^0(X)$ as the kernel of $\deg_X$.
\end{defn}

We obviously have an exact sequence
\begin{equation}\label{eq:D0}
0 \to \Div^0(X) \to \Div(X) \overset{\deg_X}{\to} \Z \to 0.
\end{equation}

\begin{defn}\label{defn:LAD}
We define $\Z$-homomorphisms
\[
\cL_X, A_X, D_X: \Div(X) \to \Div(X)
\]
by
\[
A_X([v]) = \sum_{e \in \bE_{X, v}} [t(e)],
\quad
D_X([v]) = (\# \bE_{X, v}) [v]
\]
for $v \in V_X$, and $\cL_X = D_X - A_X$.
The presentation matrices of $\cL_X$, $A_X$, and $D_X$ (with respect to the basis $\{ [v] \}_{v \in V_X}$) are called the Laplacian matrix, the adjacency matrix, and the degree matrix, respectively.
\end{defn}

It is easy to see $\deg_X \circ \cL_X = 0$, that is, the image of $\cL_X$ is contained in $\Div^0(X)$.
Therefore, we can make the following definition.

\begin{defn}\label{defn:JP}
We define the Jacobian group $\Jac(X)$ and the Picard group $\Pic(X)$ as the cokernels
\[
\Jac(X) = \Coker(\cL_X: \Div(X) \to \Div^0(X))
\]
and
\[
\Pic(X) = \Coker(\cL_X: \Div(X) \to \Div(X)).
\]
\end{defn}

Therefore, by \eqref{eq:D0}, we have an exact sequence
\begin{equation}\label{eq:JP_ex}
0 \to \Jac(X) \to \Pic(X) \overset{\deg_X}{\to} \Z\to 0.
\end{equation}

A fundamental property of the Jacobian group is the following.

\begin{thm}[Kirchhoff's matrix tree theorem]\label{thm:Kirch}
The Jacobian group $\Jac(X)$ is a finite $\Z$-module.
In fact, the cardinality of $\Jac(X)$ is equal to the number of spanning trees of $X$.
\end{thm}

Let us observe the following homological description of the Jacobian group.

\begin{lem}\label{lem:cpx_J}
Let $\iota_X: \Z \to \Div(X)$ be the $\Z$-homomorphism that sends $1$ to $\sum_{v \in V_X} [v]$.
Then the sequence
\begin{equation}\label{eq:cpx_J}
0 \to \Z \overset{\iota_X}{\to} \Div(X) \overset{\cL_X}{\to} \Div(X) \overset{\deg_X}{\to} \Z \to 0
\end{equation}
is a complex that is acyclic except for the right $\Div(X)$, at which the homology group is isomorphic to $\Jac(X)$.
\end{lem}

\begin{proof}
We only have to check $\Image(\iota_X) = \Ker(\cL_X)$; the other assertions are clear.
It is easy to see that $\cL_X \circ \iota_X = 0$, that is, $\Image(\iota_X) \subset \Ker(\cL_X)$.
Thanks to Theorem \ref{thm:Kirch}, the $\Z$-ranks of $\Image(\iota_X)$ and $ \Ker(\cL_X)$ are the same.
Moreover, the definition of $\iota_X$ implies that $\Image(\iota_X)$ is a saturated submodule of $\Div(X)$ (i.e., the cokernel of $\iota_X$ is $\Z$-torsion-free).
These observations imply $\Image(\iota_X) = \Ker(\cL_X)$.
\end{proof}

Now we quickly recall the notion of Galois coverings of connected graphs (see Terras \cite[Chapter 13]{Ter11} or \cite[\S 2.1]{RV22} for the details).
A covering $Y/X$ of connected graphs means that we are given a covering map $\pi: Y \to X$, which is surjective and locally isomorphic.
The degree of $Y/X$ is defined as $\# (\pi^{-1}(v))$ for any choice of $v \in V_X$.
The Galois group $\Gamma$ of $Y/X$ is defined as the automorphism group of $Y$ that respects the covering map.
We say that a covering $Y/X$ is Galois (or normal) if the order $\# \Gamma$ of $\Gamma$ is equal to the degree of $Y/X$.
We say $Y/X$ is abelian if $\Gamma$ is abelian.

\begin{defn}\label{defn:Z2}
Let $Y/X$ be a Galois covering of connected graphs with Galois group $\Gamma$.
It is easy to see that $\Div(Y)$ is a free $\Z[\Gamma]$-module and the endomorphism $\cL_Y$ on $\Div(Y)$ is a $\Z[\Gamma]$-homomorphism.
In case $\Gamma$ is abelian, we define
\begin{equation}\label{eq:Z_YX}
Z_{Y/X} = \det_{\Z[\Gamma]}(\cL_Y \mid \Div(Y)) \in \Z[\Gamma].
\end{equation}
\end{defn}

By the definitions of $\Pic(Y)$ and of Fitting ideals, we have
\begin{equation}\label{eq:Z_YX2}
\Fitt_{\Z[\Gamma]}(\Pic(Y)) = (Z_{Y/X})
\end{equation}
when $\Gamma$ is abelian.
In \S \ref{ss:Ihara}, we will see a relation between $Z_{Y/X}$ and the Ihara zeta function.

\section{Reduction to an algebraic problem}\label{sec:Fin_arith}

In this section, we describe the Fitting ideal that is concerned in Theorem \ref{thm:main_fin} by using a shifted Fitting ideal of an explicit module.

We begin with an elementary lemma.
Let $\Gamma$ be a finite group and write
\[
R = \Z[\Gamma],
\qquad
\ol{R} = R/(N_{\Gamma}),
\]
where we set $N_{\Gamma} = \sum_{\gamma \in \Gamma} \gamma \in \Z[\Gamma]$.

\begin{lem}\label{lem:Tor1}
We have
\[
\Tor_1^R(\ol{R}, \Z) = 0,
\quad
\ol{R} \otimes_R \Z \simeq \Z/(\# \Gamma)\Z.
\]
\end{lem}

\begin{proof}
We have an exact sequence of $R$-modules
\[
0 \to \Z \overset{\iota}{\to} R \to \ol{R} \to 0,
\]
where $\iota$ sends $1$ to $N_{\Gamma}$.
Then the associated long exact sequence of $\Tor_{*}^R(-, \Z)$ reads
\[
0 \to \Tor_1^R(\ol{R}, \Z) \to \Z \otimes_R \Z \overset{\iota_*}{\to} R \otimes_R \Z \to \ol{R} \otimes_R \Z \to 0.
\]
We have natural isomorphisms
\[
\Z \otimes_R \Z \simeq \Z,
\quad
R \otimes_R \Z \simeq \Z,
\]
through which $\iota_*$ is identified with the multiplication by $\# \Gamma$ on $\Z$.
Therefore, we obtain the lemma.
\end{proof}

In the rest of this section, let $Y/X$ be a Galois covering of connected graphs.
We write $\Gamma$ for its Galois group and use the same notation $R$ and $\ol{R}$ as above.

The following is the key ingredient to prove Theorem \ref{thm:main_fin}.

\begin{prop}\label{prop:Jac_ex1}
We have an exact sequence of finite $\ol{R}$-modules
\[
0 \to \ol{R} \otimes_R \Jac(Y) \to \ol{R} \otimes_R \Pic(Y) \to \Z/(\# \Gamma)\Z \to 0.
\]
\end{prop}

\begin{proof}
Let us consider the exact sequence \eqref{eq:JP_ex} for $Y$ instead of $X$.
By taking $\Tor_*^R(\ol{R}, -)$, we obtain an exact sequence
\[
\Tor_1^R(\ol{R}, \Z) \to \ol{R} \otimes_R \Jac(Y) \to \ol{R} \otimes_R \Pic(Y) \to \ol{R} \otimes_R \Z \to 0.
\]
By applying Lemma \ref{lem:Tor1}, we obtain the proposition.
\end{proof}

Let us study $\ol{R} \otimes_R \Pic(Y)$.

\begin{lem}\label{lem:olPic}
We have a short exact sequence
\[
0 \to\ol{R} \otimes_R \Div(Y) \overset{\ol{\cL_{Y}}}{\to} \ol{R} \otimes_R \Div(Y) \to \ol{R} \otimes_R \Pic(Y) \to 0,
\]
where $\ol{\cL_Y}$ denotes the homomorphism induced by $\cL_Y$.
\end{lem}

\begin{proof}
By the definition of $\Pic(Y)$, the cokernel of $\ol{\cL_Y}$ is isomorphic to $\ol{R} \otimes_R \Pic(Y)$, which is finite by Theorem \ref{thm:Kirch} and Proposition \ref{prop:Jac_ex1}.
Then the injectivity of $\ol{\cL_Y}$ follows since it is an endomorphism of a free $\Z$-module of finite rank.
\end{proof}

We write $\ol{Z_{Y/X}} \in \ol{R}$ for the image of $Z_{Y/X} \in R$.
By applying the shift theory (see \S \ref{ss:shift_defn}), we obtain the following.

\begin{thm}\label{thm:Fitt_fin1}
Suppose that $Y/X$ is an abelian covering.
Then we have
\[
\Fitt_{\ol{R}}(\ol{R} \otimes_R \Jac(Y))
= \ol{Z_{Y/X}} \Fitt^{[1]}_{\ol{R}}(\Z/ (\#\Gamma) \Z)
\]
as ideals of $\ol{R}$.
\end{thm}

\begin{proof}
As in \S \ref{ss:shift_defn}, let $\cP_{\ol{R}}$ be the category of finite $\ol{R}$-modules whose projective dimensions are $\leq 1$.
By Lemma \ref{lem:olPic}, we see that $\ol{R} \otimes_R \Pic(Y)$ is in $\cP_{\ol{R}}$.
Moreover, \eqref{eq:Z_YX2} implies
\[
\Fitt_{\ol{R}}(\ol{R} \otimes_R \Pic(Y)) 
= (\ol{Z_{Y/X}}).
\]
Therefore, the theorem follows from the definition of $\Fitt^{[1]}_{\ol{R}}(-)$ and Proposition \ref{prop:Jac_ex1}.
\end{proof}

We will determine $\Fitt^{[1]}_{\ol{R}}(\Z/ (\#\Gamma) \Z)$ in \S \ref{sec:fin_alg}, which would complete the proof of Theorem \ref{thm:main_fin}.

Before closing this section, we show a proposition that provides an interpretation of $\ol{R} \otimes_R \Jac(Y)$.
In the course of the proof, we reproduce the exact sequence in Proposition \ref{prop:Jac_ex1}.

\begin{prop}\label{prop:NormY}
We have a natural isomorphism
\[
\Jac(X) \simeq N_{\Gamma} \Jac(Y),
\]
so
\[
\ol{R} \otimes_R \Jac(Y) 
\simeq \Jac(Y)/\Jac(X).
\]
\end{prop}

\begin{proof}
We make use of Lemma \ref{lem:cpx_J} for both $X$ and $Y$.
Let us consider the following diagram involving those complexes.
\[
\xymatrix{
	0 \ar[r]
	& \Z \ar[r]^-{\iota_X} \ar@{=}[d]
	& \Div(X) \ar[r]^{\cL_X} \ar@{^(->}[d]
	& \Div(X) \ar[r]^-{\deg_X} \ar@{^(->}[d]
	& \Z \ar[r] \ar@{^(->}[d]^{\times (\# \Gamma)}
	& 0\\
	0 \ar[r]
	& \Z \ar[r]_-{\iota_Y} \ar[d]
	& \Div(Y) \ar[r]_{\cL_{Y}} \ar@{->>}[d]
	& \Div(Y) \ar[r]_-{\deg_{Y}} \ar@{->>}[d]
	& \Z \ar[r] \ar@{->>}[d]
	& 0\\
	& 0 \ar[r] 
	& \ol{R} \otimes_{R} \Div(Y) \ar[r]_{\ol{\cL_{Y}}}
	& \ol{R} \otimes_{R} \Div(Y) \ar[r]
	& \Z/(\# \Gamma) \Z \ar[r] 
	& 0.
}
\]
Here, the injective homomorphism $\Div(X) \hookrightarrow \Div(Y)$ is defined by sending $[v]$ to $\sum_{w \in \pi^{-1}(v)} [w]$, where $\pi: Y \to X$ is the covering map.
Then the image of this homomorphism is $N_{\Gamma} \Div(Y)$.
The vertical sequences are all exact and the lower horizontal sequence is the induced complex.

We regard this large diagram as a short exact sequence of complexes.
Recall that the kernel and the cokernel of $\ol{\cL_Y}$ are determined in Lemma \ref{lem:olPic}.
Then, taking the homology groups, we obtain an exact sequence
\begin{equation}\label{eq:31-1}
0 \to \Jac(X) \to \Jac(Y) \to \ol{R} \otimes_R \Pic(Y)  \to \Z/(\# \Gamma) \Z \to 0.
\end{equation}
The construction of the embedding of $\Jac(X)$ into $\Jac(Y)$ shows that the image coincides with $N_{\Gamma} \Jac(Y)$, as claimed.
Moreover, this reproduces the exact sequence in Proposition \ref{prop:Jac_ex1}.
\end{proof}

\section{Algebraic result}\label{sec:fin_alg}

Let $\Gamma$ be any finite abelian group.
Set $R = \Z[\Gamma]$ and $\ol{R} = \Z[\Gamma]/(N_{\Gamma})$.
The goal of this section is to determine the factor $\Fitt^{[1]}_{\ol{R}}(\Z/ (\#\Gamma) \Z)$ that appeared in Theorem \ref{thm:Fitt_fin1}.

\subsection{Statement}\label{ss:fin_alg}

First we state the result.
As in \S \ref{ss:main_result}, we fix a decomposition $\Gamma = \Delta_1 \times \cdots \times \Delta_s$ into cyclic groups.
For each $1 \leq l \leq s$, we choose a generator $\sigma_l$ of $\Delta_l$ and define $n_l = \# \Delta_l$ and elements $\tau_l, \nu_l, D_l \in \Z[\Delta_l]$ by the same formulas.
Clearly we have $\tau_l \nu_l = 0$.
A key property of $D_l$ is 
\begin{equation}\label{eq:D_key}
\tau_l D_l = n_l - \nu_l,
\end{equation}
which can be proved by a direct computation.
This property is also important in the theory of Euler systems.

The result is the following.

\begin{thm}\label{thm:Fitt_fin}
We have
\begin{align}
\Fitt^{[1]}_{\ol{R}}(\Z/ (\#\Gamma) \Z)
= & \sum_{l=1}^s \left( \nu_1 \cdots \nu_{l-1} \cdot \frac{D_l}{n_l} \cdot \nu_{l+1} \cdots \nu_s \right) \\
& + \left(\nu_1^{e_1} \cdots \nu_s^{e_s} \tau_1^{f_1} \cdots \tau_s^{f_s} \, \middle| \, \begin{matrix} 0 \leq e_l \leq 1, f_l \geq 0, \\ e_1 + \cdots + e_s + f_1 + \cdots + f_s = s- 2 \end{matrix} \right).
\end{align}
\end{thm}

Then clearly Theorem \ref{thm:main_fin} follows from Theorems \ref{thm:Fitt_fin1} and \ref{thm:Fitt_fin}.
The rest of this section is devoted to the proof of Theorem \ref{thm:Fitt_fin}.

First we describe $\Fitt^{[1]}_{\ol{R}}(\Z/ (\#\Gamma) \Z)$ by using an (unshifted) Fitting ideal.
Let 
\[
I = \Ker(\Z[\Gamma] \to \Z) \subset R
\]
 be the augmentation ideal, which we can regard as an $\ol{R}$-module.

\begin{lem}\label{lem:Fitt1}
We have
\[
\Fitt^{[1]}_{\ol{R}}(\Z/ (\#\Gamma) \Z) 
 = (\# \Gamma)^{-1} \Fitt_{\ol{R}}(I/R(\# \Gamma - N_{\Gamma})).
\]
\end{lem}

\begin{proof}
By the tautological exact sequence $0 \to I \to R \to \Z \to 0$, we have an exact sequence $0 \to I \to \ol{R} \to \Z/(\# \Gamma)\Z \to 0$.
Since $\# \Gamma - N_{\Gamma}$ is in $I$ and its image to $\ol{R}$ is $\# \Gamma$, we then obtain an exact sequence
\[
0 \to I/R (\# \Gamma - N_{\Gamma}) \to \ol{R}/ (\# \Gamma) \ol{R} \to \Z/(\# \Gamma)\Z \to 0.
\]
By the definition of $\Fitt_{\ol{R}}^{[1]}(-)$, this implies the lemma.
\end{proof}

By Lemma \ref{lem:Fitt1}, the proof of Theorem \ref{thm:Fitt_fin} is reduced to the following.

\begin{prop}\label{prop:Fitt_fin2}
We have
\begin{align}
& \Fitt_{R}(I/R(\# \Gamma - N_{\Gamma}))\\
& \qquad = 
	(N_{\Gamma})
	+ (\# \Gamma) \cdot \sum_{l=1}^s \left( \nu_1 \cdots \nu_{l-1} \cdot \frac{D_l}{n_l} \cdot \nu_{l+1} \cdots \nu_s \right) \\
& \qquad \qquad + (\# \Gamma) \cdot \left( \nu_1^{e_1} \cdots \nu_s^{e_s} \tau_1^{f_1} \cdots \tau_s^{f_s} \, \middle| \, \begin{matrix} 0 \leq e_l \leq 1, f_l \geq 0, \\ e_1 + \cdots + e_s + f_1 + \cdots + f_s = s- 2 \end{matrix} \right).
\end{align}
\end{prop}

The proof of this proposition will be given in \S \ref{ss:thm2}.
Before that, we establish a key technical proposition in \S \ref{ss:resol}.

\subsection{Key proposition}\label{ss:resol}

Let $s \geq 0$ be an integer and let $R_s = \Z[T_1, \dots, T_s]$ be the polynomial ring in $s$ indeterminates.
Let us construct a free resolution of $\Z$ over $R_s$ by using tensor products of complexes.
This, or rather the more complicated construction in \S \ref{ss:thm2} below, is inspired by \cite{GK15} and is also used in \cite{Kata_05}, \cite{AK21}, etc.

For each $1 \leq l \leq s$, we have an exact sequence
\[
0 \to \Z[T_l] x_l \overset{T_l}{\to} \Z[T_l] \to \Z \to 0.
\]
Here, $x_l$ denotes an indeterminate, so $\Z[T_l] x_l$ is a free $\Z[T_l]$-module of rank one, and the map labeled $T_l$ sends $x_l$ to $T_l$.
The map $\Z[T_l] \to \Z$ sends $T_l$ to $0$.

Observe that $R_s$ is the tensor product of $\Z[T_l]$ over $\Z$ for $1 \leq l \leq s$.
Then by taking the tensor product over $\Z$ of the complexes $[\Z[T_l] x_l \overset{T_l}{\to} \Z[T_l]]$ for $1 \leq l \leq s$, we obtain an exact sequence
\[
\bigoplus_{1 \leq l < l' \leq s} R_s x_l x_{l'} \overset{d_2}{\to} \bigoplus_{1 \leq l \leq s} R_s x_l \overset{d_1}{\to} R_s \overset{d_0}{\to} \Z \to 0,
\]
where $d_0$ sends $T_l$ to $0$ for any $l$, $d_1$ is determined by 
\[
d_1(x_l) = T_l
\]
for $1 \leq l \leq s$, and $d_2$ is determined by
\[
d_2(x_l x_{l'}) = - T_{l'} x_l + T_l x_{l'}
\]
for $1 \leq l < l' \leq s$ (there are other choices of signs, but that does not matter).

We write $N_s(T_1, \dots, T_s)$ for the presentation matrix of $d_2$ with respect to the basis $\{x_l x_{l'} \mid 1 \leq l < l' \leq s \}$ and $x_1, \dots, x_s$.
Note that we do not determine the order of the rows because it does not matter at all.
Indeed, in the following we will sometimes choose various orders of bases that are suitable for computation.

\begin{eg}
When $s = 0$ (resp.~$s = 1$), by definition the source of $d_2$ is the zero module, so $d_2 = 0$ and $N_0()$ (resp.~$N_1(T_1)$) is the empty matrix.
When $s = 2$, we have 
\[
N_2(T_1, T_2) = \begin{pmatrix} -T_2 & T_1 \end{pmatrix}.
\]
When $s = 3$, we have
\[
N_3(T_1, T_2, T_3) = 
\begin{pmatrix}
& - T_3 & T_2\\
- T_3 & & T_1\\
- T_2 & T_1 &
\end{pmatrix}.
\]
Here, we use the order $x_2x_3, x_1x_3, x_1x_2$ for the basis.
When $s = 4$, we have
\[
N_4(T_1, T_2, T_3, T_4) = 
\begin{pmatrix}
-T_2 & T_1 &  & \\
- T_3 & & T_1 & \\
- T_4 & & & T_1\\
& - T_3 & T_2 & \\
& -T_4 & & T_2\\
& & -T_4 & T_3
\end{pmatrix}.
\]
Here, we use the order $x_1 x_2, x_1x_3, x_1x_4, x_2x_3, x_2x_4, x_3x_4$ for the basis.
\end{eg}

For a matrix $H$, which is identified with a homomorphism between free modules, we have the associated ideal $\Fitt(H)$ defined as in Definition \ref{defn:Fitt1}.

The following formula plays a key role in the proof of Proposition \ref{prop:Fitt_fin2}.
It is the most technical result in this paper.

\begin{prop}\label{prop:Fitt_p2}
We have
\[
\Fitt(\begin{pmatrix} N_s(T_1, \dots, T_s) \\ \begin{matrix} B_1 & \cdots & B_s \end{matrix} \end{pmatrix})
= \begin{cases}
(1) & (s = 0)\\
(B_1) & (s = 1)\\
\left( \sum_{l=1}^s B_l T_l \right)(T_1, \dots, T_s)^{s - 2} & (s \geq 2).\\
\end{cases}
\]
Here, $B_1, \dots, B_s$ are indeterminates that have no relation with $T_1, \dots, T_s$, i.e., we work over the polynomial ring $\Z[T_1, \dots, T_s, B_1, \dots, B_s]$.
\end{prop}

The rest of this subsection is devoted to the proof of this proposition.
The cases $s = 0$, $s = 1$ are clear.
Let us suppose $s \geq 2$.

Recall that the rows of $N_s(T_1, \dots, T_s)$ are labeled $x_l x_{l'}$ ($1 \leq l < l' \leq s$).
We also attach a label $y$ to the last row $\begin{pmatrix} B_1 & \cdots & B_s \end{pmatrix}$.
For any subset 
\[
\cA \subset \{x_l x_{l'} \mid 1 \leq l < l' \leq s\} \amalg \{y\}
\]
with $\# \cA = s$, let $N_{\cA}$ be the $s \times s$ submatrix that is constructed by picking up the rows whose labels are in $\cA$.
Then the left hand side in Proposition \ref{prop:Fitt_p2} is generated by $\det(N_{\cA})$ for all such $\cA$.

For such an $\cA$, let us construct an undirected simple graph $G_{\cA}$ that has $s$ vertices $x_1, \dots, x_s$ so that $x_l x_{l'} \in \cA$ if and only if $x_l$ and $x_{l'}$ are adjacent.
Then the number of the edges of $G_{\cA}$ is either $s-1$ or $s$; indeed, it is $s-1$ if and only if $y \in \cA$.
This construction gives a one-to-one correspondence between the set
\[
\{ \cA \subset \{x_l x_{l'} \mid 1 \leq l < l' \leq s\} \amalg \{y\} \mid \# \cA = s\}
\]
and the set
\[
\{ \text{ simple graph structures on the set of vertices $\{x_1, \dots, x_s\}$ with $s-1$ or $s$ edges }\}.
\]

We shall describe $\det(N_{\cA})$ by using information about the associated graph $G_{\cA}$.

\begin{claim}\label{claim2}
We have $\det(N_{\cA}) = 0$ unless $G_{\cA}$ is a tree.
\end{claim}

\begin{proof}
Suppose that $G_{\cA}$ is not a tree.
Since the number of the edges of $G_{\cA}$ is at least $s-1$, this assumption is equivalent to that $G_{\cA}$ has a cycle.
In other words, by permutation of the indices, we may assume that $\{x_1x_2, x_2 x_3, \dots, x_{r-1} x_r, x_r x_1\} \subset \cA$ for some $3 \leq r \leq s$.
In this case, the matrix $N_{\cA}$ is of the form
\[
\begin{pmatrix}
	\begin{matrix}
		-T_2 & T_1 & & & \\
		& -T_3 & T_2 & & \\
		& & \ddots & \ddots & \\
		& & & - T_r & T_{r-1}\\
		-T_r & & & & T_1
	\end{matrix}
 & 0\\
* & * \\
\end{pmatrix}
\]
if we use the order
\[
x_1, x_2, \dots, x_r, *, \dots, *
\]
of the vertices and the order 
\[
x_1x_2, x_2 x_3, \dots, x_{r-1} x_r, x_r x_1, *, \dots, *
\]
 of the rows in $\cA$ ($*$ denotes unspecified things).
It is easy to see that the $r \times r$ matrix in the upper left has determinant $0$.
Therefore, the claim follows.
\end{proof}

By Claim \ref{claim2}, we only have to deal with the case where $G_{\cA}$ is a tree.
Note that then the number of the edges of $G_{\cA}$ is $s-1$, i.e., $y \in \cA$.
For each $1 \leq l \leq s$, let $\deg_{\cA}(x_l) \geq 1$ be the degree of $x_l$ in the graph $G_{\cA}$.
By definition, $\deg_{\cA}(x_l)$ is the number of vertices that are adjacent to $x_l$ in $G_{\cA}$.

\begin{claim}\label{claim3}
If $G_{\cA}$ is a tree, then we have
\[
\det(N_{\cA}) = \pm \left(\sum_{l = 1}^s B_l T_l \right) \cdot T_1^{\deg_{\cA}(x_1)-1} \cdots T_s^{\deg_{\cA}(x_s) - 1}.
\]
\end{claim}

\begin{proof}
Let us write simply $\deg(-)$ instead of $\deg_{\cA}(-)$.
First we show that the claim follows from the following weaker claim: there exist signs $\epsilon_1, \dots, \epsilon_s \in \{\pm 1\}$ such that we have
\[
\det(N_{\cA}) = \left(\sum_{l = 1}^s \epsilon_l B_l T_l \right) \cdot T_1^{\deg(x_1)-1} \cdots T_s^{\deg(x_s)-1}.
\]
Suppose that such a family $\{\epsilon_l\}_l$ exists.
By Claim \ref{claim2}, for any $l \neq l'$, we know that $\det(N_{\cA})$ vanishes if we set $B_l = T_{l'}$, $B_{l'} = -T_l$, and the other $B$'s to be zero.
Therefore, $\epsilon_l T_{l'} T_l + \epsilon_{l'} (-T_l) T_{l'} = 0$, which shows $\epsilon_l = \epsilon_{l'}$.
Thus we obtain $\epsilon_1 = \cdots = \epsilon_s$, so the full claim follows.

Let us show the weaker claim.
By the cofactor expansion of $\det(N_{\cA})$ with respect to the final row (labeled $y$), we have
\[
\det(N_{\cA})
= \sum_{l = 1}^s B_l \delta_l,
\]
where $\delta_l$ denotes the cofactor of $B_l$ (i.e., the determinant of the submatrix obtained by eliminating the last row $y$ and the $l$-th column).
Here and henceforth, we ignore the sign of $\delta_l$, which does not matter for the weaker claim.
Then it is enough to show that, for any fixed $l$, the cofactor $\delta_l$ coincides with $T_l \cdot T_1^{\deg(x_1)-1} \cdots T_s^{\deg(x_s)-1}$ up to sign.

Fix $l$.
Let us reorder the vertices $x_1, \dots, x_s$ as follows.
We regard $x_l$ as the root of the tree $G_{\cA}$.
Recall that then the depth of each vertex $x_k$, denoted by $\depth(x_k)$, is defined as the length of the unique path from the root $x_l$ to $x_k$.
(We set $\depth(x_l) = 0$.)
Now we reorder the vertices of the rooted graph $G_{\cA}$ as
\[
x_{\sigma(1)}, x_{\sigma(2)}, \dots, x_{\sigma(s)},
\]
where $\sigma$ is a permutation of the set $\{1, 2, \dots, s\}$ such that 
\[
\depth(x_{\sigma(k)}) \leq  \depth(x_{\sigma(k+1)})
\]
for every $1 \leq k \leq s - 1$.
We necessarily have $\sigma(1) = l$, but this $\sigma$ is not unique in general.

For each $2 \leq k \leq s$, let $1 \leq \ul{k} \leq s$ be the index such that $x_{\sigma(\ul{k})}$ is the parent of $x_{\sigma(k)}$.
It is the unique vertex that is adjacent to $x_{\sigma(k)}$ and whose depth is less than that of $x_{\sigma(k)}$.
Note that we have $1 \leq \ul{k} < k$.

Now, to compute the cofactor $\delta_l$, we use the order
\[
x_{\sigma(2)}, x_{\sigma(3)}, \dots, x_{\sigma(s)}
\]
of the vertices and the order
\[
x_{\sigma(\ul{2})} x_{\sigma(2)}, x_{\sigma(\ul{3})} x_{\sigma(3)}, \dots, x_{\sigma(\ul{s})} x_{\sigma(s)}
\]
of the edges.
Then, thanks to $\ul{k} < k$, the matrix whose determinant is $\delta_l$ is lower triangular and we obtain
\[
\delta_l = \pm T_{\sigma(\ul{2})} T_{\sigma(\ul{3})} \cdots T_{\sigma(\ul{s})}.
\]
Since $x_{\sigma(\ul{k})}$ is the parent of $x_{\sigma(k)}$, for each $1 \leq l' \leq s$, the exponent of the indeterminate $T_{l'}$ in this product is equal to the number of the children of the vertex $x_{l'}$.
If the vertex $x_{l'}$ is the root, i.e., if $l' = l$, then the number of children is equal to $\deg(x_l)$.
Otherwise, i.e., if $l' \neq l$, the number of children is equal to $\deg(x_{l'}) - 1$.
This completes the proof of the claim.
\end{proof}

\begin{eg}
We illustrate the above proof by an example.
Let $s = 5$ and consider $\cA = \{x_1 x_2, x_1 x_3, x_1 x_4, x_2 x_5, y\}$, so
\[
N_{\cA}
= \begin{pmatrix}
	-T_2 & T_1 & & &\\
	-T_3 & & T_1 & &\\
	-T_4 & & & T_1 &\\
	& -T_5 & & & T_2\\
	B_1 & B_2 & B_3 & B_4 & B_5	
\end{pmatrix}.
\]
In this case we have $\deg_{\cA}(x_1) = 3$, $\deg_{\cA}(x_2) = 2$, and $\deg_{\cA}(x_3) = \deg_{\cA}(x_4) = \deg_{\cA}(x_5) = 1$.

To consider the cofactor $\delta_1$ of $B_1$, we regard $x_1$ as the root.
Then $\depth(x_2) = \depth(x_3) = \depth(x_4) = 1$ and $\depth(x_5) = 2$, so we use the order $x_2, x_3, x_4, x_5$ for the vertices.
We have $\ul{2} = 1$, $\ul{3} = 1$, $\ul{4} = 1$, $\ul{5} = 2$, so we use the order $x_1 x_2, x_1 x_3, x_1 x_4, x_2 x_5$ for the edges.
Then we obtain
\[
\delta_1 = \pm \det 
\begin{pmatrix}
	T_1 & & &\\
	& T_1 & &\\
	& & T_1 &\\
	T_5 & & & T_2\\
\end{pmatrix}
= \pm T_1^3 T_2.
\]
Here and in the following examples, we omit writing $\pm$ before the indeterminates since we may ignore the signs.

Similarly, to consider $\delta_2$, we use $x_1, x_5, x_3, x_4$ for the vertices and $x_1 x_2, x_2 x_5, x_1 x_3, x_1 x_4$ for the edges, and obtain
\[
\delta_2 = \pm \det 
\begin{pmatrix}
	T_2 & & &\\
	& T_2 & &\\
	T_3 & & T_1 &\\
	T_4 & & & T_1\\
\end{pmatrix}
= \pm T_1^2 T_2^2.
\]

To consider $\delta_3$, we use the orders $x_1, x_2, x_4, x_5$ and $x_1 x_3, x_1 x_2, x_1 x_4, x_2 x_5$ and obtain
\[
\delta_3 = \pm \det 
\begin{pmatrix}
	T_3 & & &\\
	T_2 & T_1 & &\\
	T_4 & & T_1 &\\
	& T_5 & & T_2\\
\end{pmatrix}
= \pm T_1^2 T_2 T_3.
\]
\end{eg}

Now we return to the general case.
By Claims \ref{claim2} and \ref{claim3}, the Fitting ideal to be computed in Proposition \ref{prop:Fitt_p2} is generated by
\[
\left(\sum_{l = 1}^s B_l T_l \right) \cdot T_1^{\deg_{\cA}(x_1)-1} \cdots T_s^{\deg_{\cA}(x_s) - 1},
\]
where $G_{\cA}$ varies all tree structures on the set of vertices $\{x_1, \dots, x_s \}$.
Note that we have
\[
\sum_{l=1}^s (\deg_{\cA}(x_l) - 1) 
= \sum_{l=1}^s \deg_{\cA}(x_l) - s
= 2(s-1) - s = s-2.
\]
It remains only to show that, conversely, the tuple $(\deg_{\cA}(x_1) - 1, \dots, \deg_{\cA}(x_s) - 1)$ can be any tuple whose sum is $s - 2$.
This assertion follows from the following.

\begin{claim}
Let $s \geq 2$ and let $f_1, \dots, f_s$ be integers such that $f_l \geq 0$ and $f_1 + \dots + f_s = s - 2$.
Then there exists a tree structure on the set of vertices $x_1, \dots, x_s$ such that the degree of $x_l$ equals $f_l + 1$ for all $1 \leq l \leq s$.
\end{claim}

\begin{proof}
We argue by the induction on $s$.
When $s = 2$, we must have $f_1 = f_2 = 0$, and the unique tree structure satisfies the property.
Suppose $s \geq 3$.
By $f_1 + \dots + f_s = s - 2$, we have $f_l = 0$ and $f_{l'} \geq 1$ for some $l$ and $l'$, so we may assume that $f_s = 0$ and $f_1 \geq 1$.
By the induction hypothesis, there exists a tree structure on $x_1, \dots, x_{s-1}$ such that the degree of $x_1$ is $f_1$ and the degree of $x_l$ is $f_l+1$ for $2 \leq l \leq s-1$.
Then we obtain the desired graph by simply connecting $x_1$ and $x_s$.
\end{proof}

This completes the proof of Proposition \ref{prop:Fitt_p2}.

\subsection{Proof of Proposition \ref{prop:Fitt_fin2}}\label{ss:thm2}

In this subsection, we prove Proposition \ref{prop:Fitt_fin2}.

We first construct a free resolution of the augmentation ideal $I \subset R = \Z[\Gamma]$ over $R$ in a similar way as in \S \ref{ss:resol}.
Note that the idea is already used in previous work such as \cite{AK21}.

For each $1 \leq l \leq s$, we have an exact sequence
\[
\Z[\Delta_l] x_l^2 \overset{\nu_l}{\to} \Z[\Delta_l] x_l \overset{\tau_l}{\to} \Z[\Delta_l] \to \Z \to 0,
\]
where $x_l$ denotes an indeterminate, the map $\nu_l$ sends $x_l^2$ to $\nu_l x_l$, $\tau_l$ sends $x_l$ to $\tau_l$, and the map $\Z[\Delta_l] \to \Z$ is the augmentation map.

By taking the tensor product of the complexes $[\Z[\Delta_l] x_l^2 \overset{\nu_l}{\to} \Z[\Delta_l] x_l \overset{\tau_l}{\to} \Z[\Delta_l]]$ over $\Z$, we obtain an exact sequence
\[
\bigoplus_{l=1}^s R x_l^2 \oplus \bigoplus_{1 \leq l < l' \leq s} R x_l x_{l'} 
\overset{d_2}{\to} \bigoplus_{l=1}^s R x_l 
\overset{d_1}{\to} R
\overset{d_0}{\to} \Z \to 0,
\]
where $d_0$ is the augmentation map, $d_1$ is determined by
\[
d_1(x_l) = \tau_l  \quad (1 \leq l \leq s),
\]
and $d_2$ is determined by
\[
\begin{cases}
d_2(x_l^2) = \nu_l x_l & (1 \leq l \leq s),\\
d_2(x_l x_{l'}) = -\tau_{l'} x_l + \tau_l x_{l'} &(1 \leq l < l' \leq s).
\end{cases}
\]

Therefore, the presentation matrix of $d_2$ is of the form
\[
M_s(\nu_1,\dots, \nu_s, \tau_1, \dots, \tau_s) = 
\begin{pmatrix}
\begin{matrix}
\nu_1 & & \\
& \ddots & \\
& & \nu_s \\
\end{matrix}\\
 N_s(\tau_1, \dots, \tau_s) 
\end{pmatrix}.
\]
Here, $N_s(\tau_1, \dots, \tau_s)$ denotes the matrix obtained by setting $T_l = \tau_l$ in the matrix $N_s(T_1, \dots, T_s)$ constructed in \S \ref{ss:resol}.

We have constructed a presentation $M_s(\nu_1,\dots, \nu_s, \tau_1, \dots, \tau_s)$ of the $\Z[\Gamma]$-module $I = \Ker(d_0: R \to \Z)$.
Our next task is to construct a presentation of $I/R(\# \Gamma - N_{\Gamma})$.
For that purpose, we define elements $b_1, \cdots, b_s \in R$ by
\[
b_l = \nu_1 \cdots \nu_{l-1} \cdot D_l \cdot n_{l+1} \cdots n_s
\]
for $1 \leq l \leq s$.
This is where the Kolyvagin derivative operators come into play.

\begin{lem}\label{lem:sum_bt}
We have
\[
\sum_{l=1}^s b_l \tau_l = \# \Gamma - N_{\Gamma}.
\]
\end{lem}

\begin{proof}
By using the identity \eqref{eq:D_key}, we can compute
\begin{align}
\sum_{l=1}^s b_l \tau_l 
& = \sum_{l=1}^s \nu_1 \cdots \nu_{l-1} \cdot (n_l - \nu_l) \cdot n_{l+1} \cdots n_s\\
& = n_1 \cdots n_s - \nu_1 \cdots \nu_s\\
& = \# \Gamma - N_{\Gamma}.
\end{align}
Thus we obtain the lemma.
\end{proof}

Then Lemma \ref{lem:sum_bt} implies that
\[
d_1 \left(\sum_{l=1}^s b_l x_l \right) = \# \Gamma - N_{\Gamma},
\]
so $I/R(\# \Gamma - N_{\Gamma})$ has a presentation matrix
\[
\begin{pmatrix} M_s(\nu_1, \dots, \nu_s, \tau_1, \dots, \tau_s) \\ \begin{matrix} b_1 & \cdots & b_s \end{matrix} \end{pmatrix}
 = \begin{pmatrix} \begin{matrix} \nu_1 & & \\
& \ddots & \\
& & \nu_s \\
\end{matrix}\\
 N_s(\tau_1, \dots, \tau_s) \\
 \begin{matrix} b_1 & \cdots & b_s \end{matrix}
 \end{pmatrix}
\]
over $R$.

To  compute the Fitting ideal of this matrix,
we first observe
\begin{align}\label{eq:tM0}
& \Fitt(
	\begin{pmatrix} 
		M_s(\nu_1, \dots, \nu_s, \tau_1, \dots, \tau_s) \\ \begin{matrix} b_1 & \cdots & b_s \end{matrix}
	\end{pmatrix}
) \\
& \quad = \sum_{j = 0}^s \sum_{\substack{{\bm a} \subset \{1, 2, \dots, s\} \\ \# {\bm a} = j}} \nu_{a_1} \cdots \nu_{a_j} 
\Fitt(
	\begin{pmatrix} N_{s - j}(\tau_{a_{j + 1}}, \dots, \tau_{a_s}) \\ \begin{matrix} b_{a_{j+1}} & \cdots & b_{a_s} \end{matrix}
	\end{pmatrix}
).
\end{align}
Here, $\bm{a}$ runs over the subsets of $\{1, 2, \dots, s\}$ with $\# \bm{a} = j$ and define $a_1, \dots, a_s$ by requiring 
\[
\bm{a} = \{a_1, \dots, a_j\},
\qquad \{a_1, \dots, a_s\} = \{1, 2, \dots, s\}, 
\qquad a_1 < \dots < a_j, 
\qquad a_{j + 1} < \dots < a_s.
\]
We can show \eqref{eq:tM0} directly by using the identity $\nu_l \tau_l = 0$ (see \cite[Proposition 4.7]{AK21} for a similar reasoning).

Let us compute the right hand side of \eqref{eq:tM0} by using Proposition \ref{prop:Fitt_p2}.
When $j = s$, we have $\bm{a} = \{1, 2, \dots, s\}$, so the term is 
\begin{equation}\label{eq:tM1}
\nu_1 \cdots \nu_s (1) = (N_{\Gamma}).
\end{equation}
When $j = s - 1$, we have $\{1, 2, \dots, s\} \setminus \bm{a} = \{ l \}$ for some $1 \leq l \leq s$, and then the term is
\begin{equation}\label{eq:tM2}
\nu_1 \cdots \nu_{l-1} \cdot \nu_{l+1} \cdots \nu_s \cdot (b_l)
= (\# \Gamma) \cdot \left( \nu_1 \cdots \nu_{l-1} \cdot \frac{D_l}{n_l} \cdot \nu_{l+1} \cdots \nu_s \right).
\end{equation}
Finally we consider $0 \leq j \leq s - 2$.
For each $\bm{a}$ with $\# \bm{a} = j$, the term can be computed as
\begin{align}\label{eq:tM3}
& \nu_{a_1} \cdots \nu_{a_j} \left( b_{a_{j+1}} \tau_{a_{j+1}} + \cdots + b_{a_s} \tau_{a_s} \right)(\tau_{a_{j+1}}, \dots, \tau_{a_s})^{s - j - 2}\\
& \quad = \nu_{a_1} \cdots \nu_{a_j} \left( b_{a_1} \tau_{a_1} + \cdots + b_{a_s} \tau_{a_s} \right)(\tau_{a_{1}}, \dots, \tau_{a_s})^{s - j - 2}\\
& \quad = \nu_{a_1} \cdots \nu_{a_j} ( \# \Gamma - N_{\Gamma}) (\tau_1, \dots, \tau_s)^{s - j - 2},
\end{align}
where the first equality follows from the identity $\tau_l \nu_l = 0$; the second from Lemma \ref{lem:sum_bt}.

It is easy to see that the ideal generated by \eqref{eq:tM1}, \eqref{eq:tM2}, and \eqref{eq:tM3} coincides with the right hand side of the proposition (observe that, thanks to \eqref{eq:tM1}, the $N_{\Gamma}$ in \eqref{eq:tM3} can be ignored).
This completes the proof of Proposition \ref{prop:Fitt_fin2}.

This also completes the proof of Theorem \ref{thm:main_fin}.

\section{Voltage graphs and derived graphs}\label{sec:pre}

To formulate the setup for infinite coverings, it is convenient to use the notion of voltage graphs and their derived graphs.

\subsection{Voltage graphs and derived graphs}\label{ss:vol_graph}

See \cite[\S 2.3]{Gon22}, \cite[\S 4]{MV21b}, or \cite[\S 2.3]{RV22} for more details.

\begin{defn}
A voltage graph $(X, \Gamma, \alpha)$ consists of a graph $X$, a group $\Gamma$, and a map $\alpha: \bE_X \to \Gamma$ satisfying $\alpha(\ol{e}) = \alpha(e)^{-1}$ for any $e \in \bE_X$.
We do not assume that $\Gamma$ is finite or abelian unless explicitly stated.
\end{defn}

Let $(X, \Gamma, \alpha)$ be a voltage graph.

\begin{defn}
Suppose that $\Gamma$ is finite.
Then we construct a graph $X(\Gamma)$ (the map $\alpha$ is implicit), called the derived graph of $(X, \Gamma, \alpha)$, by
\[
V_{X(\Gamma)} = \Gamma \times V_X,
\quad
\bE_{X(\Gamma)} = \Gamma \times \bE_X,
\]
and
\[
\ol{(\gamma, e)} = (\gamma \cdot \alpha(e), \ol{e}),
\quad
s_{X(\Gamma)}((\gamma, e)) = (\gamma, s_X(e)),
\quad
t_{X(\Gamma)}((\gamma, e)) = (\gamma \cdot \alpha(e), t_X(e))
\]
for any $(\gamma, e) \in \Gamma \times \bE_X$.
\end{defn}

The group $\Gamma$ naturally acts on the graph $X(\Gamma)$ from the left.
Moreover, $X(\Gamma)$ is a covering of $X$ via the natural projection and the action of $\Gamma$ respects this covering structure.

If $X(\Gamma)$ is connected (so $X$ is also connected), then $X(\Gamma)$ is actually a Galois covering of $X$ whose Galois group is $\Gamma$.
Conversely, if we are given a Galois covering $Y/X$ of connected graphs, then there exists a voltage graph structure $(X, \Gamma, \alpha)$ with $\Gamma$ the Galois group such that $X(\Gamma)$ and $Y$ are isomorphic as coverings of $X$.
In a nutshell, the notion of derived graphs covers the notion of Galois coverings of connected graphs.

\begin{defn}\label{defn:AGamma}
We define $\Z[\Gamma]$-homomorphisms
\[
A_{X, \Gamma}, \cL_{X, \Gamma}: \Z[\Gamma] \otimes_{\Z} \Div(X) \to \Z[\Gamma] \otimes_{\Z} \Div(X)
\]
by
\[
A_{X, \Gamma}(1 \otimes [v]) = \sum_{e \in \bE_{X, v}} \alpha(e) \otimes [t(e)]
\]
for any $v \in V_X$ and
\[
\cL_{X, \Gamma} = \id \otimes D_X - A_{X, \Gamma},
\]
where $D_X: \Div(X) \to \Div(X)$ is as in Definition \ref{defn:LAD}.
Note that when $\Gamma$ is trivial, these maps $A_{X, \Gamma}$ and $\cL_{X, \Gamma}$ are naturally identified with the maps $A_X$ and $\cL_X$ in Definition \ref{defn:LAD}.
\end{defn}

The following lemma is easily proved.

\begin{lem}\label{lem:Pic_DA}
Suppose that $\Gamma$ is finite. 
Then we have a commutative diagram of $\Z[\Gamma]$-modules
\[
\xymatrix{
	\Z[\Gamma] \otimes_{\Z} \Div(X) \ar[r]^{\cL_{X, \Gamma}} \ar[d]_{\simeq}
	& \Z[\Gamma] \otimes_{\Z} \Div(X) \ar[d]^{\simeq}\\
	\Div(X(\Gamma)) \ar[r]_{\cL_{X(\Gamma)}}
	& \Div(X(\Gamma)),
}
\]
where the vertical isomorphisms are defined by sending $\gamma \otimes [v]$ to $[(\gamma, v)]$.
In particular, $\Pic(X(\Gamma))$ is isomorphic to the cokernel of the homomorphism $\cL_{X, \Gamma}$.
\end{lem}

\begin{defn}\label{defn:Z}
Suppose that $\Gamma$ is abelian.
We define
\begin{equation}\label{eq:Zdet1}
Z_{X, \Gamma} = \det_{\Z[\Gamma]} ( \cL_{X, \Gamma} \mid \Z[\Gamma] \otimes_{\Z} \Div(X)) \in \Z[\Gamma].
\end{equation}
\end{defn}

Suppose that $\Gamma$ is finite and abelian and that $X(\Gamma)$ is connected.
Then Definition \ref{defn:Z2} gives us an element $Z_{X(\Gamma)/X} \in \Z[\Gamma]$.
It is related to the element $Z_{X, \Gamma}$ simply by
\[
Z_{X(\Gamma)/X} = Z_{X, \Gamma},
\]
thanks to Lemma \ref{lem:Pic_DA}.

\subsection{Profinite coverings}\label{ss:profin}

In this subsection, we consider a voltage graph $(X, \Gamma, \alpha)$ such that $\Gamma$ is profinite.

For each open normal subgroup $U$ of $\Gamma$, we have the voltage graph $(X, \Gamma/U, \alpha_{/U})$, where $\alpha_{/U}$ is the composite map of $\alpha$ and the natural projection $\Gamma \to \Gamma/U$.
Therefore, we have the associated derived graph $X(\Gamma/U)$, on which $\Gamma/U$ acts.

Even though $X(\Gamma)$ is not defined unless $\Gamma$ is finite, let us write $X(\Gamma)$ to mean the family $\{X(\Gamma/U)\}_U$.
Then $X(\Gamma)$ can be regarded as an infinite covering of $X$.
For instance, suppose that $\Gamma$ is isomorphic to $\Z_p = \varprojlim_n \Z/p^n\Z$ as a topological group.
Then the open subgroups of $\Gamma$, written multiplicatively, are $\Gamma^{p^n}$ for each $n \geq 0$.
Thus $X_{\infty} = X(\Gamma)$ is the collection of $X_n = X(\Gamma/\Gamma^{p^n})$ for $n \geq 0$ in this case.
This family is illustrated as a tower of coverings
\[
X = X_0 \leftarrow X_1 \leftarrow X_2 \leftarrow \cdots.
\]
We call such an $X_{\infty}/X$ a $\Z_p$-covering.
This is regarded as an analogue of $\Z_p$-extensions of number fields in Iwasawa theory; see \S \ref{sec:Iwa} for more on this theme.

For simplicity, in the rest of this subsection, we will always assume the following:

\begin{ass}\label{ass:conn}
For any open normal subgroup $U$ of $\Gamma$, the derived graph $X(\Gamma/U)$ is connected.
\end{ass}

As we will review in the proof of the following lemma, we have an equivalent condition for the derived graph to be connected, and that  imposes a restriction to the structure of $\Gamma$ as follows.

\begin{lem}\label{lem:decomp_fg}
If Assumption \ref{ass:conn} holds, then the group $\Gamma$ is finitely generated as a profinite group.
Indeed, it is generated by $\# E_X - \# V_X + 1$ elements (recall $\# E_X = \frac{1}{2} \cdot \# \bE_X$).
\end{lem}

\begin{proof}
Let $\pi_1(X, v_0)$ denote the fundamental group of $X$ with an arbitrarily fixed base point $v_0 \in V_X$.
It is known that $\pi_1(X, v_0)$ is a free group on $\# E_X - \# V_X + 1$ elements.
Moreover, the condition that the derived graph $X(\Gamma/U)$ is connected is equivalent to that the group homomorphism $\pi_1(X, v_0) \to \Gamma/U$ induced by $\alpha_{/U}$ is surjective (see \cite[Theorem 2.11]{RV22} for instance).
This implies that Assumption \ref{ass:conn} is equivalent to that the image of the group homomorphism $\pi_1(X, v_0) \to \Gamma$ is dense in $\Gamma$.
Therefore, we obtain the lemma.
\end{proof}

In order to guarantee the exactness of inverse limits, we change the coefficient ring from $\Z$ to a compact $\Z$-algebra $\Lambda$ that is flat over $\Z$ (i.e., torsion-free as a $\Z$-module).
Fundamental examples of $\Lambda$ include $\hZ$ and $\Z_p$ for a prime number $p$, where $\hZ$ (resp.~$\Z_p$) is the profinite (resp.~pro-$p$) completion of $\Z$.

\begin{defn}
We define 
\[
\Jac_{\Lambda}(X(\Gamma)) = \varprojlim_{U} \left(\Lambda \otimes_{\Z} \Jac(X(\Gamma/U)) \right),
\]
where $U$ runs over the open normal subgroups of $\Gamma$.
This is a module over the completed group ring
\[
\Lambda[[\Gamma]] = \varprojlim_U \Lambda[\Gamma/U].
\]
We also define $\Pic_{\Lambda}(X(\Gamma))$ in the same way.
\end{defn}

Let us observe several propositions about these Jacobian groups and Picard groups.

\begin{prop}\label{prop:JP_ex_inf}
We have an exact sequence of $\Lambda[[\Gamma]]$-modules
\[
0 \to \Jac_{\Lambda}(X(\Gamma)) \to \Pic_{\Lambda}(X(\Gamma)) \to \Lambda \to 0.
\]
\end{prop}

\begin{proof}
For each open normal subgroup $U$ of $\Gamma$, we have the exact sequence \eqref{eq:JP_ex} for $X(\Gamma/U)$.
Then, as $\Lambda$ is compact, by taking the limit (after base change from $\Z$ to $\Lambda$), we obtain the claimed exact sequence.
\end{proof}

We define an ideal $\GamLambda$ of $\Lambda$ by
\[
\GamLambda = \bigcap_{U} [\Gamma: U] \Lambda,
\]
where $U$ runs over all open normal subgroups of $\Gamma$.
We often have $\GamLambda = 0$ (see Remark \ref{rem:Lap_inj}).

Recall that we have the endomorphism $\cL_{X, \Gamma}$ on $\Z[\Gamma] \otimes_{\Z} \Div(X)$ (Definition \ref{defn:AGamma}).
Let the same symbol denote the base change to $\Lambda[[\Gamma]] \otimes_{\Z} \Div(X)$.
Also recall that $Z_{X, \Gamma} \in \Z[\Gamma] \subset \Lambda[[\Gamma]]$ is defined in Definition \ref{defn:Z}.

\begin{prop}\label{prop:Pic_ex_inf}
We have an exact sequence of $\Lambda[[\Gamma]]$-modules
\begin{equation}\label{eq:Pic_ex}
0 \to \GamLambda \to \Lambda[[\Gamma]] \otimes_{\Z} \Div(X) \overset{\cL_{X, \Gamma}}{\to} \Lambda[[\Gamma]] \otimes_{\Z} \Div(X)
\to \Pic_{\Lambda}(X(\Gamma)) \to 0.
\end{equation}
In particular, when $\Gamma$ is abelian, we have
\[
\Fitt_{\Lambda[[\Gamma]]}(\Pic_{\Lambda}(X(\Gamma)))
= (Z_{X, \Gamma}).
\]
\end{prop}

\begin{proof}
For each open normal subgroup $U$ of $\Gamma$, we may apply Lemmas \ref{lem:cpx_J} and \ref{lem:Pic_DA} for $X(\Gamma/U)$.
Since $\Lambda$ is flat over $\Z$, we obtain an exact sequence
\[
0 \to \Lambda \overset{\iota_{X, \Gamma/U}}{\to} \Lambda[\Gamma/U] \otimes_{\Z} \Div(X) \overset{\cL_{X, \Gamma/U}}{\to} \Lambda[\Gamma/U] \otimes_{\Z} \Div(X) \to \Pic_{\Lambda}(X(\Gamma/U)) \to 0,
\]
where $\iota_{X, \Gamma/U}$ denotes the $\Lambda$-homomorphism that sends $1$ to $N_{\Gamma/U} \otimes \left(\sum_{v \in V_X} [v] \right)$ with $N_{\Gamma/U}$ the norm element.
When we are given another open normal subgroup $V$ of $\Gamma$ such that $V \subset U$, the exact sequences satisfy the following commutative diagram
\[
\xymatrix{
	0 \ar[r]
	& \Lambda \ar[r]^-{\iota_{X, \Gamma/V}} \ar@{^(->}[d]_-{\times [U:V]}
	& \Lambda[\Gamma/V] \otimes_{\Z} \Div(X) \ar[r]^-{\cL_{X, \Gamma/V}} \ar@{->>}[d]
	& \Lambda[\Gamma/V] \otimes_{\Z} \Div(X) \ar[r] \ar@{->>}[d]
	& \Pic_{\Lambda}(X(\Gamma/V)) \ar[r] \ar@{->>}[d]
	& 0\\
	0 \ar[r]
	& \Lambda \ar[r]^-{\iota_{X, \Gamma/U}}
	& \Lambda[\Gamma/U] \otimes_{\Z} \Div(X) \ar[r]^-{\cL_{X, \Gamma/U}}
	& \Lambda[\Gamma/U] \otimes_{\Z} \Div(X) \ar[r]
	& \Pic_{\Lambda}(X(\Gamma/U)) \ar[r]
	& 0.
}
\]
As $\Lambda$ is compact, by taking the limit, we obtain the claimed exact sequence.
\end{proof}

\begin{cor}\label{cor:Pic_coinv}
Let $\Gamma'$ be a closed normal subgroup of $\Gamma$ and consider the induced voltage graph $(X, \Gamma/\Gamma', \alpha_{/\Gamma'})$, where $\alpha_{/\Gamma'}$ denotes the composition of $\alpha$ and the projection map $\Gamma \to \Gamma/\Gamma'$.
Then we have an isomorphism of $\Lambda[[\Gamma/\Gamma']]$-modules
\[
\Pic_{\Lambda}(X(\Gamma))_{\Gamma'} \simeq \Pic_{\Lambda}(X(\Gamma/\Gamma')).
\]
\end{cor}

\begin{proof}
This follows immediately by comparing the exact sequences obtained by Proposition \ref{prop:Pic_ex_inf} applied to the voltage graphs $(X, \Gamma, \alpha)$ and $(X, \Gamma/\Gamma', \alpha_{/\Gamma'})$.
\end{proof}

Recall that a module over a commutative ring is said to be torsion if any element is annihilated by a non-zero-divisor.
By Proposition \ref{prop:Pic_ex_inf}, taking Lemma \ref{lem:tors_gen} into account, we immediately obtain the following.

\begin{cor}\label{cor:inj}
The following are equivalent.
\begin{itemize}
\item[(i)]
We have $\GamLambda = 0$.
\item[(ii)]
The endomorphism $\cL_{X, \Gamma}$ on $\Lambda[[\Gamma]] \otimes_{\Z} \Div(X)$ is injective.
\item[(iii)]
The element $Z_{X, \Gamma}$ of $\Lambda[[\Gamma]]$ is a non-zero-divisor.
\item[(iv)]
The $\Lambda[[\Gamma]]$-module $\Pic_{\Lambda}(X(\Gamma))$ is torsion.
\end{itemize}
\end{cor}

\begin{rem}\label{rem:Lap_inj}
In \S \ref{sec:Fitt_infin}, we basically consider the case where the equivalent conditions in Corollary \ref{cor:inj} hold.
For instance, when $\Lambda = \hZ$, we have $\GamLambda = 0$ if and only if the order of $\Gamma$ is divisible by any positive integers.
Here, in general, we say that the order of a profinite group $\Gamma$ is divisible by a positive integer $n$ if there exists an open normal subgroup $U \subset \Gamma$ such that $n \mid [\Gamma: U]$.
Similarly, when $\Lambda = \Z_p$, we have $\GamLambda = 0$ if and only if the order of $\Gamma$ is divisible by any $p$-power (i.e., divisible by $p^{\infty}$).
\end{rem}

Now let us mention that we can apply our observations so far to naturally reprove the analogue of the Iwasawa main conjecture for multiple $\Z_p$-coverings that is shown by Kleine--M\"uller \cite{KM22}, as follows.

\begin{thm}[{Kleine--M\"uller \cite[Theorem 5.2 and Remark 5.3]{KM22}}]\label{thm:IMC}
Let us assume that $\Lambda = \Z_p$ and $\Gamma$ is isomorphic to $\Z_p^d$ for some $d \geq 1$.
Then the characteristic ideal of $\Jac_{\Z_p}(X(\Gamma))$ is described as
\[
\cha_{\Z_p[[\Gamma]]}(\Jac_{\Z_p}(X(\Gamma)))
= \begin{cases}
(Z_{X, \Gamma}) & (d \geq 2)\\
I(\Gamma)^{-1} (Z_{X, \Gamma}) & (d = 1),
\end{cases}
\]
where $I(\Gamma)$ denotes the augmentation ideal of $\Z_p[[\Gamma]]$.
\end{thm}

\begin{proof}
The coefficient ring $\Z_p[[\Gamma]]$ is known to be isomorphic to the ring of power series in $d$ indeterminates $\Z_p[[T_1, \dots, T_d]]$.
In particular, $\Z_p[[\Gamma]]$ is a regular local ring, so we have the notion of characteristic ideals $\cha_{\Z_p[[\Gamma]]}(M)$ for finitely generated torsion modules $M$.

By Corollary \ref{cor:inj}, together with Remark \ref{rem:Lap_inj}, the $\Z_p[[\Gamma]]$-module $\Pic_{\Z_p}(X(\Gamma))$ is torsion.
Since the characteristic ideals satisfy the multiplicativity for exact sequences, Proposition \ref{prop:JP_ex_inf} implies
\[
\cha_{\Z_p[[\Gamma]]}(\Pic_{\Z_p}(X(\Gamma)))
= \cha_{\Z_p[[\Gamma]]}(\Z_p) \cdot \cha_{\Z_p[[\Gamma]]}(\Jac_{\Z_p}(X(\Gamma))) .
\]
We have $\cha_{\Z_p[[\Gamma]]}(\Z_p) = (1)$ if $d \geq 2$ (since $\Z_p$ is a pseudo-null module) and $\cha_{\Z_p[[\Gamma]]}(\Z_p) = I(\Gamma)$ if $d = 1$.
Moreover, Proposition \ref{prop:Pic_ex_inf} implies
\[
\cha_{\Z_p[[\Gamma]]}(\Pic_{\Z_p}(X(\Gamma)))
= (Z_{X, \Gamma}).
\]
Thus we obtain the theorem.
\end{proof}

\begin{rem}\label{rem:F_vs_ch}
In \S \ref{sec:Fitt_infin}, we will compute the Fitting ideals rather than the characteristic ideals.
Compared to the characteristic ideals, the notion of Fitting ideals has roughly two advantages.
One is that Fitting ideals are defined over any commutative rings.
The other is that, even when the coefficient ring is a regular local ring, Fitting ideals are more refined than characteristic ideals; indeed, the Fitting ideal determines the characteristic ideal.
Therefore, the main result in \S \ref{sec:Fitt_infin} (stated as Theorems \ref{thm:F_J_inf} and \ref{thm:Fitt_infin}) is a refinement of Theorem \ref{thm:IMC}.
\end{rem}

\subsection{Relation with Ihara zeta functions}\label{ss:Ihara}

In this subsection, we briefly explain the analytic aspect of Jacobian groups.
As illustrated in Theorem \ref{thm:IMC} for instance, the structure of the Jacobian group is closely related to the element $Z_{X, \Gamma}$.
The purpose of this subsection is to explain that $Z_{X, \Gamma}$ has an interpretation using the (equivariant) Ihara zeta function.
This fact may be regarded as an analogue of the conjectural relations between algebraic aspects and analytic aspects for various arithmetic objects in number theory (e.g., the Iwasawa main conjecture, the equivariant Tamagawa number conjecture, etc.).
The results in this subsection are not used in the other parts of this paper.
A nice reference is Terras \cite{Ter11}.

Let $(X, \Gamma, \alpha)$ be a voltage graph such that $\Gamma$ is abelian.
First we define the associated zeta function (see \cite[Chapters 2 and 18]{Ter11}).

\begin{defn}\label{defn:zeta}
We define the (equivariant) Ihara zeta function by
\[
\zeta_{X, \Gamma}(u) = \prod_{[P]} ( 1- \alpha(P) u^{\nu(P)})^{-1} \in \Z[\Gamma][[u]],
\]
where $P$ runs over primitive paths in $X$, $[P]$ denotes the rotation class of $P$, and $\nu(P)$ denotes the length of $P$.
When the path $P$ consists of edges $e_1, \dots, e_n$ ($e_i \in \bE_X$ with $n = \nu(P)$), we define $\alpha(P) = \alpha(e_1) \cdots \alpha(e_n)$.
\end{defn}

Let us define
\[
Z_{X, \Gamma}(u) = \det_{\Z[\Gamma]} ( 1- A_{X, \Gamma} u + (D_X - 1) u^2 \mid \Z[\Gamma] \otimes_{\Z} \Div(X)) \in \Z[\Gamma][u].
\]
By definition we have $Z_{X, \Gamma}(1) = Z_{X, \Gamma}$.
The result is the following.

\begin{thm}[Three-term determinant formula]\label{thm:three}
We have
\[
\zeta_{X, \Gamma}(u)^{-1} = (1- u^2)^{\# E_X - \# V_X} Z_{X, \Gamma}(u).
\]
In particular, $\zeta_{X, \Gamma}(u)$ is a rational function.
\end{thm}

\begin{proof}
In case $\Gamma$ is trivial, the theorem is \cite[Theorem 2.5]{Ter11}.
More generally, the three-term determinant formula for Artin--Ihara $L$-functions is proved in \cite[Theorem 18.15]{Ter11}.
We can prove our theorem by imitating those proofs.
Alternatively, it is possible to deduce our theorem from the formula for Artin--Ihara $L$-functions.
Indeed, the general statement can be reduced to the case where $\Gamma$ is finite and $X(\Gamma)$ is connected.
In that case, the statement follows by combining the formula for Artin--Ihara $L$-functions for all the characters of $\Gamma$.
We omit the details.
\end{proof}

\section{Results for profinite coverings}\label{sec:Fitt_infin}

Let $(X, \Gamma, \alpha)$ be a voltage graph such that $\Gamma$ is profinite and abelian.
We always suppose Assumption \ref{ass:conn}, so Lemma \ref{lem:decomp_fg} implies $\Gamma$ is finitely generated.

Let $\Lambda$ be a compact flat $\Z$-algebra.
In this section, we aim at determining
\[
\Fitt_{\Lambda[[\Gamma]]}(\Jac_{\Lambda}(X(\Gamma))),
\]
assuming the equivalent conditions in Corollary \ref{cor:inj}.

For simplicity, in what follows we assume that $\Lambda = \Z_p$ with $p$ a fixed prime number.
This is the case of the most interest.
Then the case where $\Lambda = \hZ$ would also be completed because we have $\hZ = \prod_p \Z_p$.
Recall that, as we observed in Remark \ref{rem:Lap_inj}, the equivalent conditions in Corollary \ref{cor:inj} for $\Lambda = \Z_p$ says that the order of $\Gamma$ is divisible by $p^{\infty}$.

\subsection{Statement}\label{ss:inf_result}

First we reduce the question to an algebraic problem.

\begin{thm}\label{thm:F_J_inf}
Suppose that the order of $\Gamma$ is divisible by $p^{\infty}$.
Then we have
\[
\Fitt_{\Z_p[[\Gamma]]}(\Jac_{\Z_p}(X(\Gamma))) 
 = Z_{X, \Gamma} \Fitt^{[1]}_{\Z_p[[\Gamma]]}(\Z_p).
\]
\end{thm}

\begin{proof}
By the assumption, Proposition \ref{prop:Pic_ex_inf} implies that $\Pic_{\Z_p}(X(\Gamma))$ is in the category $\cP_{\Z_p[[\Gamma]]}$ defined in \S \ref{ss:shift_defn} and moreover its Fitting ideal is $(Z_{X, \Gamma})$.
Since $\Gamma$ is finitely generated, it is not hard to show that the $\Z_p[[\Gamma]]$-module $\Z_p$ is finitely presented by constructing an explicit resolution as in \S \ref{ss:inf_alg} below.
Therefore, $\Fitt^{[1]}_{\Z_p[[\Gamma]]}(\Z_p)$ is well-defined.
Moreover, by Proposition \ref{prop:JP_ex_inf}, we see that $\Jac_{\Z_p}(X(\Gamma))$ is finitely generated and at the same time the theorem follows from the definition of $\Fitt^{[1]}_{\Z_p[[\Gamma]]}(-)$.
\end{proof}

Now we are led to computation of $\Fitt^{[1]}_{\Z_p[[\Gamma]]}(\Z_p)$.
The argument in the rest of this section is valid for any finitely generated profinite abelian group $\Gamma$ whose order is divisible by $p^{\infty}$.

Since $\Gamma$ is finitely generated, we can decompose $\Gamma$ as
\begin{equation}\label{eq:dec_inf}
\Gamma = \Delta_1 \times \cdots \times \Delta_s \times \Gamma_{s+1} \times \cdots \times \Gamma_{s+t+1},
\end{equation}
where $\Delta_l$ ($1 \leq l \leq s$) is a procyclic group whose order is not divisible by $p^{\infty}$ and $\Gamma_l$ ($s+1 \leq l \leq s+t+1$) is a procyclic group whose order is divisible by $p^{\infty}$.
Note that then the integer $t + 1$ is equal to the $\Z_p$-rank of the pro-$p$ completion of $\Gamma$, so by the assumption we have $t \geq 0$.
This decomposition is not unique in general.

For $1 \leq l \leq s$, let $\sigma_l$ be a (topological) generator of $\Delta_l$ and put $\tau_l = \sigma_l - 1$.
Similarly, for $s + 1 \leq l \leq s + t + 1$, let $\sigma_l$ be a generator of $\Gamma_l$ and put $T_l = \sigma_l - 1$.

Let $1 \leq l \leq s$.
We also fix an open subgroup $U_l \subset \Delta_l$ such that the order of $U_l$ is not divisible by $p$.
Put $n_l = [\Delta_l: U_l]$.
For each open subgroup $U$ of $U_l$, put
\[
\nu_{l, U} = \frac{1}{[U_l: U]} N_{\Delta_l/U} \in \Z_p[\Delta_l/U].
\]
Then $\nu_{l, U}$ is compatible if we vary $U$, so the family defines an element
\[
\nu_l = (\nu_{l, U})_U \in \Z_p[[\Delta_l]].
\]
This element satisfies $\nu_l^2 = n_l \nu_l$ and $\tau_l \nu_l = 0$ since each $\nu_{l, U}$ satisfies these properties.
Note that though $n_l$ and $\nu_l$ depend on the choice of $U_l$, they are well-defined up to $\Z_p^{\times}$.

The following is the main result.

\begin{thm}\label{thm:Fitt_infin}
The fractional ideal $\Fitt^{[1]}_{\Z_p[[\Gamma]]}(\Z_p)$ is equal to
\[
(T_{s+1}^{-1} \nu_1 \cdots \nu_s) + \left( \nu_1^{e_1} \cdots \nu_s^{e_s} \tau_1^{f_1} \cdots \tau_s^{f_s} T_{s+1}^{f_{s+1}}
	\, \middle| \, \begin{matrix} 0 \leq e_l \leq 1, f_l \geq 0, \\ e_1 + \cdots + e_s + f_1 + \cdots + f_{s+1} = s - 1 \end{matrix} \right)
\]
if $t = 0$ and 
\[
\left( \nu_1^{e_1} \cdots \nu_s^{e_s} \tau_1^{f_1} \cdots \tau_s^{f_s} T_{s+1}^{f_{s+1}} \cdots T_{s+t+1}^{f_{s+t+1}}
	\, \middle| \, \begin{matrix} 0 \leq e_l \leq 1, f_l \geq 0, \\ e_1 + \cdots + e_s + f_1 + \cdots + f_{s+t+1} = s + t - 1 \end{matrix} \right)
\]
if $t \geq 1$.
\end{thm}

\subsection{Proof of Theorem \ref{thm:Fitt_infin}}\label{ss:inf_alg}

First we recall a result of Atsuta and the author \cite{AK21} that plays a key role in the proof.

\begin{prop}[{\cite[Proposition 4.8]{AK21}}]\label{prop:Fitt_p1}
Let $s \geq 0$ be an integer.
Let us consider the matrix $N_s(T_1, \dots, T_s)$ constructed in \S \ref{ss:resol} whose components are in the polynomial ring $\Z[T_1, \dots, T_s]$.
For each $i \geq 0$, the $i$-th Fitting ideal of the matrix is described as
\[
\Fitt_{i}(N_s(T_1, \dots, T_s))
= \begin{cases}
	(1) & (i \geq s)\\
	0 & (s \geq 1, i = 0)\\
	(T_1, \dots, T_s)^{s-i} & (1 \leq i < s).
\end{cases}
\]
\end{prop}

Put 
\[
\Gamma^0 = \Delta_1 \times \cdots \times \Delta_s \times \Gamma_{s+1} \times \cdots \times \Gamma_{s+t},
\]
so $\Gamma = \Gamma^0 \times \Gamma_{s+t+1}$.
Let us begin with constructing a free resolution of $\Z_p$ over $\Z_p[[\Gamma^0]]$ in a similar way as in \S \S \ref{ss:resol}--\ref{ss:thm2}.
Observe that we have an exact sequences
\[
\Z_p[[\Delta_l]] x_l^2 \overset{\nu_l}{\to} \Z_p[[\Delta_l]] x_l \overset{\tau_l}{\to} \Z_p[[\Delta_l]] \to \Z_p \to 0
\]
for $1 \leq l \leq s$ and 
\[
0 \to \Z_p[[\Gamma_l]] x_l \overset{T_l}{\to} \Z_p[[\Gamma_l]] \to \Z_p \to 0
\]
for $s+1 \leq l \leq s+t$.
Then, by taking the tensor product of complexes, we obtain an exact sequence
\[
\bigoplus_{l=1}^s \Z_p[[\Gamma^0]] x_l^2 \oplus 
	\bigoplus_{1 \leq l < l' \leq s+t} \Z_p[[\Gamma^0]] x_l x_{l'}
\overset{d_2}{\to} \bigoplus_{l=1}^{s+t} \Z_p[[\Gamma^0]] x_l
\overset{d_1}{\to} \Z_p[[\Gamma^0]] 
\overset{d_0}{\to} \Z_p \to 0,
\]
where $d_0$ is the augmentation map, $d_1$ is determined by
\[
d_1(x_l) = 
\begin{cases}
\tau_l & (1 \leq l \leq s)\\
T_l & (s + 1 \leq l \leq s + t),
\end{cases}
\]
and $d_2$ is determined by
\[
d_2(x_l^2) = \nu_l x_l \quad (1 \leq l \leq s)
\]
and
\[
d_2(x_l x_{l'}) = 
\begin{cases}
- \tau_{l'} x_l + \tau_l x_{l'} & (1 \leq l < l' \leq s)\\
- T_{l'} x_l + \tau_l x_{l'} & (1 \leq l \leq s < s+1 \leq l' \leq s+t)\\
- T_{l'} x_l + T_l x_{l'} & (s+1 \leq l < l' \leq s+t).
\end{cases}
\]

Therefore, the presentation matrix of $d_2$ is of the form
\[
M_{s, t}(\nu_1,\dots, \nu_s, \tau_1, \dots, \tau_s, T_{s+1}, \dots, T_{s+t}) = 
\begin{pmatrix}
	\begin{matrix}
	\begin{matrix}
	\nu_1 & & \\
	& \ddots & \\
	& & \nu_s \\
	\end{matrix}
	& 0\\
	\end{matrix}\\
	N_{s+t}(\tau_1, \dots, \tau_s, T_{s+1}, \dots, T_{s+t}) 
\end{pmatrix}.
\]
Now observe that $\Z_p[[\Gamma^0]]$ is in $\cP_{\Z_p[[\Gamma]]}$ and $\Fitt_{\Z_p[[\Gamma]]}(\Z_p[[\Gamma^0]]) = (T_{s+t+1})$.
Since we have an exact sequence
\[
0 \to \Cok(d_2) \overset{d_1}{\to} \Z_p[[\Gamma^0]] 
\overset{d_0}{\to} \Z_p \to 0,
\]
we obtain
\begin{align}
\Fitt^{[1]}_{\Z_p[[\Gamma]]}(\Z_p)
& = T_{s+t+1}^{-1} \Fitt_{\Z_p[[\Gamma]]}(\Cok(d_2)) \\
& = T_{s+t+1}^{-1} \Fitt
	( \begin{pmatrix} M_{s, t}(\nu_1,\dots, \nu_s, \tau_1, \dots, \tau_s, T_{s+1}, \dots, T_{s+t}) \\ 
		\begin{matrix}
		T_{s+t+1} & & \\ & \ddots & \\ & & T_{s+t+1}
		\end{matrix}
	\end{pmatrix} )\\
& = \sum_{i=0}^{s+t} T_{s+t+1}^{i-1} \Fitt_i(M_{s, t}(\nu_1,\dots, \nu_s, \tau_1, \dots, \tau_s, T_{s+1}, \dots, T_{s+t})).
\end{align}

For each $0 \leq i \leq s+t$, by an analogous consideration as \eqref{eq:tM0}, we have
\begin{align}\label{eq:Mst}
& \Fitt_i(M_{s, t}(\nu_1,\dots, \nu_s, \tau_1, \dots, \tau_s, T_{s+1}, \dots, T_{s+t}))\\
& \quad = \sum_{j = 0}^{\min\{s, s+t-i\}} \sum_{\substack{{\bm a} \subset \{1, 2, \dots, s\} \\ \# {\bm a} = j}} \nu_{a_1} \cdots \nu_{a_j} 
	\Fitt_i(N_{s + t - j}(\tau_{a_{j + 1}}, \dots, \tau_{a_s}, T_{s+1}, \dots, T_{s+t})).
\end{align}
Here, we use the same notation as \eqref{eq:tM0}.
Let us compute the right hand side of \eqref{eq:Mst} by using Proposition \ref{prop:Fitt_p1}.

First we consider $i = 0$.
Then by Proposition \ref{prop:Fitt_p1}, only $j = s + t$ contributes to the sum, which is possible only when $t = 0$.
Therefore, the right hand side of \eqref{eq:Mst} vanishes unless $t = 0$.
If $t = 0$, the right hand side equals $(\nu_1 \cdots \nu_s)$ (the term for $j = s$ and $\bm{a} = \{1, 2, \dots, s\}$).

Suppose $1 \leq i \leq s+t$.
Then Proposition \ref{prop:Fitt_p1} tells us that the right hand side of \eqref{eq:Mst} is equal to
\begin{align}
& \sum_{j = 0}^{\min\{s, s+t-i\}} \sum_{\substack{{\bm a} \subset \{1, 2, \dots, s\} \\ \# {\bm a} = j}} \nu_{a_1} \cdots \nu_{a_j} (\tau_{a_{j + 1}}, \dots, \tau_{a_s}, T_{s+1}, \dots, T_{s+t})^{s + t - i - j}\\
& \qquad = \sum_{j = 0}^{\min\{s, s+t-i\}} \sum_{\substack{{\bm a} \subset \{1, 2, \dots, s\} \\ \# {\bm a} = j}} \nu_{a_1} \cdots \nu_{a_j} (\tau_1, \dots, \tau_s, T_{s+1}, \dots, T_{s+t})^{s + t - i - j}\\
& \qquad = \left( \nu_1^{e_1} \cdots \nu_s^{e_s} \tau_1^{f_1} \cdots \tau_s^{f_s} T_{s+1}^{f_{s+1}} \cdots T_{s+t}^{f_{s+t}}
	\, \middle| \, \begin{matrix} 0 \leq e_l \leq 1, f_l \geq 0, \\ e_1 + \cdots + e_s + f_1 + \cdots f_{s+t} = s + t - i \end{matrix} \right).
\end{align}

Therefore, taking the summation for $1 \leq i \leq s + t$ gives
\begin{align}
& \sum_{i=1}^{s+t} T_{s+t+1}^{i-1} \Fitt_i(M_{s, t}(\nu_1,\dots, \nu_s, \tau_1, \dots, \tau_s, T_{s+1}, \dots, T_{s+t}))\\
& \qquad = \left( \nu_1^{e_1} \cdots \nu_s^{e_s} \tau_1^{f_1} \cdots \tau_s^{f_s} T_{s+1}^{f_{s+1}} \cdots T_{s+t+1}^{f_{s+t+1}}
	\, \middle| \, \begin{matrix} 0 \leq e_l \leq 1, f_l \geq 0, \\ e_1 + \cdots + e_s + f_1 + \cdots f_{s+t+1} = s + t - 1 \end{matrix} \right).
\end{align}

By combining these formulas, we obtain Theorem \ref{thm:Fitt_infin}.

\subsection{Essentially finite case}\label{ss:remark}

Let $(X, \Gamma, \alpha)$ be a voltage graph such that $\Gamma$ is profinite abelian and we suppose Assumption \ref{ass:conn}.
Let us fix a prime number $p$

In this section so far, we studied the case where the order of $\Gamma$ is divisible by $p^{\infty}$.
In case the order of $\Gamma$ is not divisible by $p^{\infty}$ (but possibly infinite), we can obtain the following results, which are profinite generalizations of the results in \S \S \ref{sec:Fin_arith}--\ref{sec:fin_alg}.
We state the results without proof because it is essentially the same as in \S \S \ref{sec:Fin_arith}--\ref{sec:fin_alg}.

Let us construct a decomposition
\[
\Gamma = \Delta_1 \times \cdots \times \Delta_s,
\]
where $\Delta_l$ is a procyclic group (whose order is necessarily not divisible by $p^{\infty}$).
We fix an open subgroup $U_l \subset \Delta_l$ such that the order of $U_l$ is not divisible by $p$.
Then we introduce $\tau_l \in \Z_p[\Delta_l]$, a positive integer $n_l$, and $\nu_l \in \Z_p[[\Delta_l]]$ in the same way as in \S \ref{ss:inf_result}.
We then define $\nu_{\Gamma} = \nu_1 \cdots \nu_l \in \Z_p[[\Gamma]]$.

Then, as an analogue of Theorem \ref{thm:Fitt_fin1}, we have
\[
\Fitt_{\Z_p[[\Gamma]]/(\nu_{\Gamma})}(\Jac_{\Z_p}(X(\Gamma))/\nu_{\Gamma} \Jac_{\Z_p}(X(\Gamma)))
= \ol{Z_{X, \Gamma}} \Fitt^{[1]}_{\Z_p[[\Gamma]]/(\nu_{\Gamma})}(\Z_p/ (n_1 \cdots n_s) \Z_p).
\]
Moreover, as an analogue of Theorem \ref{thm:Fitt_fin}, we have
\begin{align}
& \Fitt^{[1]}_{\Z_p[[\Gamma]]/(\nu_{\Gamma})}(\Z_p/ (n_1 \cdots n_s) \Z_p)\\
& \quad =  \sum_{l=1}^s \left( \nu_1 \cdots \nu_{l-1} \cdot \frac{D_l}{n_l} \cdot \nu_{l+1} \cdots \nu_s \right) \\
& \qquad + \left(\nu_1^{e_1} \cdots \nu_s^{e_s} \tau_1^{f_1} \cdots \tau_s^{f_s} \, \middle| \, \begin{matrix} 0 \leq e_l \leq 1, f_l \geq 0, \\ e_1 + \cdots + e_s + f_1 + \cdots + f_s = s- 2 \end{matrix} \right).
\end{align}
Here, $D_l \in \Z_p[[\Delta_l]]$ is any element such that \eqref{eq:D_key} holds.

\section{Self-duality of Jacobian groups}\label{sec:dual}

The prerequisite for this section is \S \ref{sec:defn_J_graph}.
In this section, we observe a self-duality property of the Jacobian groups.
This is in contrast to the corresponding story in number theory as discussed in Remark \ref{rem:NT_dual} below.

First we fix our convention about duals.
For a finite $\Z$-module $M$, we define its Pontryagin dual by $M^{\vee} = \Hom_{\Z}(M, \Q/\Z)$.
Note that we have an alternative description $M^{\vee} = \Ext^1_{\Z}(M, \Z)$.
For a finitely generated $\Z$-module $M$, we also define its $\Z$-linear dual by $M^* = \Hom_{\Z}(M, \Z)$.
Both $(-)^{\vee}$ and $(-)^*$ are contravariant functors.

If we have a left (resp.~right) action of a group $\Gamma$ on such an $M$, we introduce a right (resp.~left) action of $\Gamma$ on $M^{\vee}$ or $M^*$ by
\[
(\phi \gamma)(x) = \phi(\gamma x)
\quad
\text{(resp.~$(\gamma \phi)(x) = \phi(x \gamma)$)}
\]
for $\gamma \in \Gamma$, $\phi \in M^{\vee}$ or $\phi \in M^*$, and $x \in M$.
Let $\iota$ be the involution on $\Gamma$ defined by $\iota(\gamma) = \gamma^{-1}$.
If $M$ is a left (resp.~right) $\Gamma$-module, we define a right (resp.~left) $\Gamma$-module $M^{\iota}$ by $M^{\iota} = M$ as an additive module and the group action is defined by $x \gamma = \gamma^{-1} x$ (resp.~$\gamma x = x \gamma^{-1}$) for $\gamma \in \Gamma$ and $x \in M$.

Now let $Y$ be a connected graph equipped with a left action of a group $\Gamma$.
Since the set of vertices $\{ [w] \}_{w \in V_Y}$ is a $\Z$-basis of $\Div(Y)$, we can construct the dual basis $\{\phi_w\}_{w \in V_Y}$ of $\Div(Y)^*$.
Namely, $\phi_w$ is the homomorphism characterized by $\phi_w([w]) = 1$ and $\phi_w([w']) = 0$ for any $w' \neq w$.
It is easy to check that $\phi_{\gamma w} = \phi_{w} \gamma^{-1}$ for any $w \in V_Y$ and $\gamma \in \Gamma$.
Therefore, we have a $\Z$-isomorphism of left $\Z[\Gamma]$-modules
\begin{equation}\label{eq:Div_dual}
\Div(Y) \simeq \Div(Y)^{*, \iota}
\end{equation}
by sending $[w]$ to $\phi_w$.

Now let us consider the following commutative diagram
\begin{equation}\label{eq:comp_L_dual2}
\xymatrix{
	0 \ar[r]
	& \Z \ar[r]^-{\iota_Y} \ar[d]_{\simeq}
	& \Div(Y) \ar[r]^-{\cL_Y} \ar[d]_{\simeq}
	& \Div(Y) \ar[r]^-{\deg_Y} \ar[d]_{\simeq}
	& \Z \ar[r] \ar[d]_{\simeq}
	& 0\\
	0 \ar[r]
	& \Z^{*, \iota} \ar[r]_-{\deg_Y^*}
	& \Div(Y)^{*, \iota} \ar[r]_-{\cL_Y^*}
	& \Div(Y)^{*, \iota} \ar[r]_-{\iota_Y^*}
	&\Z^{*, \iota} \ar[r]
	& 0.
}
\end{equation}
Here, the upper sequence is the complex constructed in Lemma \ref{lem:cpx_J}, and the lower is the $\Z$-linear dual of the upper.
The middle two vertical isomorphisms are \eqref{eq:Div_dual}.
The left and right vertical isomorphisms are the natural ones.
It is straightforward to prove that this diagram is commutative.

\begin{lem}\label{lem:dual_cpx}
The lower complex in \eqref{eq:comp_L_dual2} is acyclic except for the right $\Div(Y)^{*, \iota}$, where the homology group is isomorphic to $\Jac(Y)^{\vee, \iota}$.
\end{lem}

\begin{proof}
By Lemma \ref{lem:cpx_J}, the upper complex is quasi-isomorphic to the finite module $\Jac(Y)$ (located at the appropriate degree).
Therefore, its $\Z$-linear dual is quasi-isomorphic to $\Ext^1_{\Z}(\Jac(Y), \Z) \simeq \Jac(Y)^{\vee}$.
This proves the assertion.
In a more elementary way, it is possible to deduce the assertion by splitting the upper complex to three short exact sequences and then taking their $\Z$-linear duals (we omit the details).
\end{proof}

Now we obtain the following.

\begin{prop}\label{prop:dual}
We have a natural isomorphism
\[
\Jac(Y) \simeq \Jac(Y)^{\vee, \iota}
\]
as left $\Z[\Gamma]$-modules.
\end{prop}

\begin{proof}
This follows immediately from Lemma \ref{lem:dual_cpx}.
\end{proof}

\begin{rem}\label{rem:NT_dual}
Proposition \ref{prop:dual} implies 
\[
\Fitt_{\Z[\Gamma]}(\Jac(Y)^{\vee}) = \iota \left( \Fitt_{\Z[\Gamma]}(\Jac(Y))\right)
\]
as long as $\Gamma$ is commutative.
This phenomenon is in contrast to the situation in number theory.
As revealed in \cite{AK21} (see Remark \ref{rem:AK}), the Fitting ideal of $\Cl_L^{T, -}$ is much more complicated than that of $\Cl_L^{T, -, \vee}$.
\end{rem}

\section{Iwasawa theory for graphs from a module-theoretic viewpoint}\label{sec:Iwa}

The prerequisite for this section is \S \S \ref{sec:defn_J_graph} and \ref{sec:pre}.

In this section, we discuss Iwasawa theory for graphs.
In \S \ref{ss:ICNF} (resp.~\S \ref{ss:Kida}), we give a short proof of an analogue of the Iwasawa class number formula (resp.~of Kida's formula) for graphs.
The results are not new and indeed already obtained by others; the Iwasawa class number formula is proved independently by Gonet \cite{Gon22} and McGown--Valli\`{e}res \cite{MV21b}, and Kida's formula is proved by Ray--Valli\`eres \cite{RV22}.
However, the proofs in this paper are different from the previous ones to some extent.
Our observation is that those results directly follow from rather general module-theoretic propositions.
An advantage of this is, for instance, that we can avoid separate discussion on the case where the Euler characteristic of the base graph is zero as in previous works.
The author thinks that this method is more concise and so it is worth publishing.

\subsection{Iwasawa class number formula for graphs}\label{ss:ICNF}

Let us briefly review the original Iwasawa class number formula, proved by Iwasawa \cite[Theorem 11]{Iwa59}.
Let $K_{\infty}/K$ be a $\Z_p$-extension of number fields.
This means that $K_{\infty}/K$ is a Galois extension whose Galois group $\Gamma$ is isomorphic to $\Z_p$.
For each $n \geq 0$, let $K_n$ be the intermediate field corresponding to $\Gamma^{p^n} \subset \Gamma$.
Then $K_n/K$ is a Galois extension whose Galois group is $\Gamma/\Gamma^{p^n} \simeq \Z/p^n\Z$ and we obtain a tower of number fields
\[
K = K_0 \subset K_1 \subset K_2 \subset \cdots.
\]
Now the Iwasawa class number formula states that there exist integers $\lambda \geq 0$, $\mu \geq 0$, and $\nu$ such that
\[
\ord_p(\# \Cl(K_n)) = \lambda n + \mu p^n + \nu
\]
for $n \gg 0$.
Here, $\Cl(K_n)$ denotes the ideal class group of $K_n$ and $\ord_p$ denotes the additive $p$-adic valuation normalized so that $\ord_p(p) = 1$.

The proof of this  formula is explained in Washington \cite[\S 13.3]{Was97}.
In general, associated to a finitely generated torsion $\Z_p[[\Gamma]]$-module $M$ are integers $\lambda(M) \geq 0$, $\mu(M) \geq 0$ that are respectively called the $\lambda$-, $\mu$-invariants of $M$.
The integers $\lambda$, $\mu$ in the Iwasawa class number formula are exactly the $\lambda$-, $\mu$-invariants of the associated Iwasawa module.

The invariants $\lambda(M)$, $\mu(M)$ are defined using the structure theorem for modules over $\Z_p[[\Gamma]]$.
We do not review the details in this paper (see, e.g., \cite[\S 13.2]{Was97} or \cite[(5.1.10)]{NSW08} for the structure theorem and \cite[(5.3.9)]{NSW08} for the definitions of $\lambda$-, $\mu$-invariants).
Let us just recall that we have $\mu(M) = 0$ if and only if $M$ is finitely generated over $\Z_p$, in which case $\lambda(M)$ is equal to the $\Z_p$-rank of $M$.

Now we consider the analogue for graphs.
Let $(X, \Gamma, \alpha)$ be a voltage graph such that $\Gamma$ is isomorphic to $\Z_p$.
We suppose Assumption \ref{ass:conn}.
Then, as introduced in \S \ref{ss:profin}, we have a $\Z_p$-covering $X_{\infty} = X(\Gamma)$ of $X$.
For each $n \geq 0$, we put $X_n = X(\Gamma/\Gamma^{p^n})$, which is called the $n$-th layer of $X_{\infty}/X$.

\begin{thm}[{Gonet \cite[Theorem 1.1]{Gon22}, McGown--Valli\`{e}res \cite[Theorem 6.1]{MV21b}}]\label{thm:ICNF}
There exist integers 
\[
\lambda = \lambda(X_{\infty}/X) \geq 0, 
\quad
\mu = \mu(X_{\infty}/X) \geq 0,
\quad
\nu = \nu(X_{\infty}/X)
\]
 such that we have
\[
\ord_p(\# \Jac(X_n)) = \lambda n + \mu p^n + \nu
\]
for $n \gg 0$.
Moreover, we have
\[
\lambda(X_{\infty}/X) = \lambda(\Jac_{\Z_p}(X_{\infty})) = \lambda(\Pic_{\Z_p}(X_{\infty})) - 1
\]
and
\[
\mu(X_{\infty}/X) = \mu(\Jac_{\Z_p}(X_{\infty})) = \mu(\Pic_{\Z_p}(X_{\infty})).
\]
\end{thm}

To prove this theorem, we apply the following algebraic proposition.

\begin{prop}\label{prop:CNF1}
Let $\Gamma$ be a profinite group that is isomorphic to $\Z_p$.
Let $M$ be a finitely generated torsion $\Z_p[[\Gamma]]$-module.
Then there exist integers $\lambda \geq 0$, $\mu \geq 0$, and $\nu$ such that 
\[
\ord_p(\# (M_{\Gamma^{p^n}}[p^{\infty}])) = \lambda n + \mu p^n + \nu
\]
for $n \gg 0$.
Here, $M_{\Gamma^{p^n}}$ denotes the coinvariant module and $(-)[p^{\infty}]$ denotes the $p$-power torsion submodule.
Indeed, the integers $\lambda$ and $\mu$ are determined as
\[
\lambda = \lambda(M) - \max_{n \geq 0} \rank_{\Z_p}(M_{\Gamma^{p^n}}),
\qquad
\mu = \mu(M).
\]
\end{prop}

\begin{proof}
See Greenberg \cite[item (4) in page 79]{Gree99}.
A key ingredient is the structure theorem for modules over $\Z_p[[\Gamma]]$.
\end{proof}

\begin{proof}[Proof of Theorem \ref{thm:ICNF}]
Let us observe the following:
\begin{itemize}
\item
For any $n \geq 0$, we have $\rank_{\Z_p}(\Pic_{\Z_p}(X_n)) = 1$ and $\Jac_{\Z_p}(X_n) = \Pic_{\Z_p}(X_n)[p^{\infty}]$ by \eqref{eq:JP_ex} and Theorem \ref{thm:Kirch}.
\item
For any $n \geq 0$, we have $\Pic_{\Z_p}(X_n) \simeq \Pic_{\Z_p}(X_{\infty})_{\Gamma^{p^n}}$ by Corollary \ref{cor:Pic_coinv}.
\item
$\Pic_{\Z_p}(X_{\infty})$ is a finitely generated torsion $\Z_p[[\Gamma]]$-module by Corollary \ref{cor:inj}.
\end{itemize}
Then, applying Proposition \ref{prop:CNF1} to $M = \Pic_{\Z_p}(X_{\infty})$, we find integers $\lambda \geq 0$, $\mu \geq 0$, and $\nu$ such that $\ord_p(\# \Jac(X_n)) = \lambda n + \mu p^n + \nu$ for $n \gg 0$.
Moreover, we have 
\[
\lambda = \lambda(\Pic_{\Z_p}(X_{\infty})) - \max_{n \geq 0} \rank_{\Z_p}(\Pic_{\Z_p}(X_{\infty})_{\Gamma^{p^n}})
\]
and $\mu = \mu(\Pic_{\Z_p}(X_{\infty}))$.
Since we have 
\[
\rank_{\Z_p}(\Pic_{\Z_p}(X_{\infty})_{\Gamma^{p^n}})
= \rank_{\Z_p}(\Pic_{\Z_p}(X_n))
= 1
\]
for any $n \geq 0$, we obtain $\lambda = \lambda(\Pic_{\Z_p}(X_{\infty})) - 1$.
Then Proposition \ref{prop:JP_ex_inf} implies that these integers $\lambda$, $\mu$ coincide respectively with the $\lambda$-, $\mu$-invariants of $\Jac_{\Z_p}(X_{\infty})$.
This completes the proof.
\end{proof}

\begin{rem}
By Proposition \ref{prop:Pic_ex_inf}, the characteristic ideal of $\Pic_{\Z_p}(X_{\infty})$ is equal to $(Z_{X, \Gamma})$.
Therefore, the invariants $\lambda$ and $\mu$ in Theorem \ref{thm:ICNF} are determined by $Z_{X, \Gamma}$.
Note that, as we observed in \S \ref{ss:Ihara}, the three-term determinant formula provides an analytic interpretation of the element $Z_{X, \Gamma}$.
\end{rem}

\subsection{Kida's formula for graphs}\label{ss:Kida}

Let us go on to Kida's formula for graphs.
We do not review the original Kida's formula for ideal class groups, proved by Kida \cite{Kid80}, and instead refer to \cite{RV22} for the literatures.

Let $\wtil{\Gamma}$ be a profinite group that is decomposed as
\[
\wtil{\Gamma} = G \times \Gamma,
\]
where $G$ is a finite $p$-group (that is not necessarily abelian) and $\Gamma$ is isomorphic to $\Z_p$.
Let $(X, \wtil{\Gamma}, \alpha)$ be a voltage graph.
We suppose Assumption \ref{ass:conn} for $(X, \wtil{\Gamma}, \alpha)$.

Now $(X, \wtil{\Gamma}, \alpha)$ induces a voltage graph $(X, \Gamma, \alpha_{/G})$, where $\alpha_{/G}$ is the composite map of $\alpha$ and the natural projection $\wtil{\Gamma} \to \Gamma$.
Therefore, we obtain a $\Z_p$-covering $X_{\infty} = X(\Gamma)$ of $X$ and can consider its $\lambda$-, $\mu$-invariants $\lambda(X_{\infty}/X)$, $\mu(X_{\infty}/X)$ as in Theorem \ref{thm:ICNF}.

On the other hand, as discussed in \cite[\S 3]{RV22}, we also obtain a $\Z_p$-covering $\wtil{X}_{\infty} = X(\wtil{\Gamma})$ of $\wtil{X} = X(G)$.
Therefore, we also have the $\lambda$-, $\mu$-invariants $\lambda(\wtil{X}_{\infty}/\wtil{X})$, $\mu(\wtil{X}_{\infty}/\wtil{X})$.

This situation should be roughly regarded as a $G$-covering of $\Z_p$-coverings.
Indeed, the graph $\wtil{X}$ is a $G$-covering of $X$, and more generally the $n$-th layer of $\wtil{X}_{\infty}/\wtil{X}$ is a $G$-covering of the $n$-th layer of $X_{\infty}/X$.
Conversely, any $G$-covering of $\Z_p$-coverings of connected graphs in this sense can be constructed in this way (see \cite[\S 3]{RV22} for more precise discussion).

Now we state the analogue of Kida's formula.

\begin{thm}[{Ray--Valli\`eres \cite[Theorem 4.1]{RV22}}]\label{thm:Kida}
We have $\mu(\wtil{X}_{\infty}/\wtil{X}) = 0$ if and only if $\mu(X_{\infty}/X) = 0$.
If these equivalent conditions hold, we have
\[
\lambda(\wtil{X}_{\infty}/\wtil{X}) + 1 = \# G \cdot (\lambda(X_{\infty}/X) + 1).
\]
\end{thm}

To prove this theorem, we use the following key algebraic proposition.

\begin{prop}\label{prop:Kida_alg}
Let $G$ be a finite $p$-group.
The following hold.
\begin{itemize}
\item[(1)]
A $\Z_p[G]$-module $M$ is finitely generated over $\Z_p$ if and only if so is the coinvariant module $M_G$.
\item[(2)]
Let $\wtil{\Gamma} = G \times \Gamma$ with $\Gamma$ isomorphic to $\Z_p$.
Let $M$ a finitely generated torsion $\Z_p[[\wtil{\Gamma}]]$-module whose projective dimension over $\Z_p[[\wtil{\Gamma}]]$ is $\leq 1$.
Suppose moreover that the equivalent conditions in (1) hold.
Then we have
\[
\rank_{\Z_p}(M) = \# G \cdot \rank_{\Z_p}(M_G).
\]
\end{itemize}
\end{prop}

\begin{proof}
(1)
The ``only if'' part is clear.
The ``if'' part follows from Nakayama's lemma since $\Z_p[G]$ is a local ring.

(2)
Let us show that the assumptions imply that $M$ is a free $\Z_p[G]$-module of finite rank.
Then the assertion would follow immediately.

It is well-known that $\pd_{\Z_p[[\Gamma]]}(M) \leq 1$ is equivalent to that $M$ has no nonzero finite $\Z_p[[\Gamma]]$-submodules (e.g., \cite[(5.3.19)]{NSW08}).
Then being finitely generated over $\Z_p$ implies that $M$ is a free $\Z_p$-module of finite rank.
On the other hand, the finiteness of $\pd_{\Z_p[[\Gamma]][G]}(M)$ implies that $M$ is $G$-cohomologically trivial.
These observations show that $M$ is a free $\Z_p[G]$-module of finite rank (e.g., \cite[(5.2.21)]{NSW08}), as claimed.
\end{proof}

\begin{proof}[Proof of Theorem \ref{thm:Kida}]
Recall that Proposition \ref{prop:Pic_ex_inf} and Corollary \ref{cor:inj} show that $\Pic_{\Z_p}(\wtil{X}_{\infty})$ is a finitely generated torsion $\Z_p[[\wtil{\Gamma}]]$-module whose projective dimension is $\leq 1$.
We also have $\Pic_{\Z_p}(X_{\infty}) \simeq \Pic_{\Z_p}(\wtil{X}_{\infty})_G$ by Corollary \ref{cor:Pic_coinv}.
Therefore, we can apply Proposition \ref{prop:Kida_alg}(1)(2) to $\Pic_{\Z_p}(X_{\infty})$.
As a result, we have $\mu(\Pic_{\Z_p}(\wtil{X}_{\infty})) = 0$ if and only if $\mu(\Pic_{\Z_p}(X_{\infty})) = 0$ and, if these equivalent conditions hold, we have
\[
\lambda(\Pic_{\Z_p}(\wtil{X}_{\infty})) = \# G \cdot \lambda(\Pic_{\Z_p}(X_{\infty})).
\]
Since we have $\lambda(X_{\infty}/X) = \lambda(\Pic_{\Z_p}(X_{\infty})) - 1$ and $\mu(X_{\infty}/X) = \mu(\Pic_{\Z_p}(X_{\infty}))$, and similarly for $\wtil{X}_{\infty}/\wtil{X}$, this proves Theorem \ref{thm:Kida}.
\end{proof}

\appendix

\section{Notes on Fitting ideals}\label{sec:Fitt_gen}

In this section, we briefly review the definition of Fitting ideals and the shift theory introduced in \cite{Kata_05}.

Let $R$ be a commutative ring.
Note that we do not assume $R$ is noetherian.
This is because the coefficient rings arising from the arithmetic in this paper are not noetherian in general (e.g., $\hZ[[\hZ]]$).
In \cite{Kata_05}, the author developed the shift theory only over noetherian rings.
In this section, we observe that the argument can be generalized to non-noetherian rings by imposing appropriate finiteness properties on modules.

\subsection{Fitting ideals}\label{ss:Fitt_defn}

We recall the definition of Fitting ideals.
See, e.g., Northcott \cite[\S 3.1]{Nor76}.

\begin{defn}\label{defn:Fitt1}
Let $I$ be a finite set and $J$ a (not necessarily finite) set.
Let
\[
h: R^{\oplus J} \to R^{\oplus I}
\]
be an $R$-homomorphism, where $R^{\oplus I}$, $R^{\oplus J}$ denote the free modules on the set $I$, $J$ respectively.

We define the (initial) Fitting ideal $\Fitt(h) = \Fitt_R(h)$ as follows.
For each subset $J' \subset J$ with $\# J' = \# I$, we write 
\[
h_{J'}: R^{\oplus J'} \to R^{\oplus I}
\]
for the restriction of $h$.
Let $\det(h_{J'})$ be the determinant of $h_{J'}$ with respect to any choice of bases; inevitably this determinant has ambiguity up to $R^{\times}$, but this does not matter.
Then $\Fitt_R(h)$ is defined as the ideal of $R$ generated by $\det(h_{J'})$ for all $J' \subset J$ with $\# J' = \# I$.
Note that if $\# J < \# I$, we have $\Fitt_R(h) = 0$ as there is no such a $J'$.

More generally, for an integer $i \geq 0$, the $i$-th Fitting ideal $\Fitt_i(h) = \Fitt_{i, R}(h)$ is defined as follows.
For subsets $I' \subset I$ and $J' \subset J$ with $\# I' = \# J'$, we write 
\[
h_{J', I'}: R^{\oplus J'} \to R^{\oplus I'}
\]
for the map induced by $h$.
If $i \leq \# I$, we define $\Fitt_{i, R}(h)$ as the ideal of $R$ generated by $\det(h_{J', I'})$ for all $J' \subset J$, $I' \subset I$ with $\# J' = \# I' = \# I - i$.
If $i > \# I$, we set $\Fitt_{i, R}(h) = R$.
\end{defn}

\begin{defn}\label{defn:Fitt2}
Let $M$ be a finitely generated $R$-module.
We define $\Fitt_R(M)$ and $\Fitt_{i, R}(M)$ ($i \geq 0$) respectively as $\Fitt_R(h)$ and $\Fitt_{i, R}(h)$, where $h$ is a homomorphism as in Definition \ref{defn:Fitt1} such that the cokernel of $h$ is isomorphic to $M$.
It is known that this definition is independent from the choice of $h$.
\end{defn}

\subsection{Shifts of Fitting ideals}\label{ss:shift_defn}

We write $\Frac(R)$ for the total ring of fractions of $R$.
An $R$-module $M$ is said to be torsion if any element is annihilated by a non-zero-divisor of $R$; equivalently if $\Frac(R) \otimes_R M = 0$.

\begin{lem}\label{lem:tors_gen}
Let $F$ be a free $R$-module of finite rank and $h: F \to F$ be an endomorphism.
Then the following are equivalent.
\begin{itemize}
\item[(i)]
The map $h$ is injective.
\item[(ii)]
The determinant $\det(h) \in R$ is a non-zero-divisor.
\item[(iii)]
The cokernel $\Coker(h)$ is a torsion $R$-module.
\end{itemize}
\end{lem}

\begin{proof}
The fact (i) $\Leftrightarrow$ (ii) is shown in \cite[Chapter III, \S 8, Proposition 3, p.~524]{BouAlg74}.
Since $\Cok(h)$ is annihilated by $\det(h)$, we have (ii) $\Rightarrow$ (iii).
If (iii) holds, the base change of $h$ from $R$ to $\Frac(R)$ is surjective, so it is isomorphic, which implies (i).
\end{proof}

Let us define $\cP_R$ as the category of finitely presented torsion $R$-module $P$ such that $\pd_R(P) \leq 1$, where $\pd_R(-)$ denotes the projective dimension.
By Lemma \ref{lem:tors_gen}, if a module $P$ satisfies an exact sequence of the form
\begin{equation}\label{eq:pres}
0 \to F \to F \to P \to 0
\end{equation}
with $F$ a free module of finite rank, then $P$ is in $\cP_R$.
Conversely, if $R$ is local, then every $P$ in $\cP_R$ satisfies an exact sequence of the form \eqref{eq:pres}, since in that case any projective module is necessarily free.

A fractional ideal $I$ of $R$ is defined as an $R$-submodule of $\Frac(R)$ such that $u I \subset R$ for some non-zero-divisor $u \in R$.

\begin{lem}\label{lem:FittP}
The following are true.
\begin{itemize}
\item[(1)]
For each $P \in \cP_R$, the Fitting ideal $\Fitt_R(P)$ is invertible as a fractional ideal of $R$.
\item[(2)]
Let $0 \to M' \to M \to P \to 0$ be an exact sequence of finitely generated torsion $R$-modules such that $P \in \cP_R$.
Then we have
\[
\Fitt_R(M) = \Fitt_R(P) \Fitt_R(M').
\]
\end{itemize}
\end{lem}

\begin{proof}
We sketch the proof (see \cite[Proposition 2.7]{Kata_05} for the details).
For both claims (1) and (2), it is enough to show them after localization at all prime ideals of $R$, so we may assume $R$ is local.
Then $P$ satisfies an exact sequence of the form \eqref{eq:pres}.
Now claim (1) follows from Lemma \ref{lem:tors_gen} and claim (2) follows by using the horseshoe lemma.
\end{proof}

Now we introduce the ``$n$-th shift'' $\Fitt^{[n]}_R(-)$ as in \cite[Theorem 2.6]{Kata_05}.
First we consider the ``once-shift'' $\Fitt^{[1]}_R(-)$, which actually suffices for the purpose of this paper.
For completeness, in Definition \ref{defn:shift_n}, we will also introduce the general $n$-th shift.

\begin{defn}
For a finitely presented torsion $R$-module $M$, we define a fractional ideal $\Fitt^{[1]}_R(M)$ as follows.
Let us take an exact sequence of $R$-modules
\begin{equation}\label{eq:resol1}
0 \to N \to P \to M \to 0
\end{equation}
with $P \in \cP_R$.
For instance, if $M$ is generated by $n$ elements and annihilated by a non-zero-divisor $f \in R$, we may take $P = (R/fR)^n$.
Since $M$ is finitely presented, $N$ is finitely generated, so its Fitting ideal is defined.
Then we define (using Lemma \ref{lem:FittP}(1))
\[
\Fitt^{[1]}_R(M) = \Fitt_R(P)^{-1} \Fitt_R(N).
\]
This is well-defined, that is, independent from the choice of \eqref{eq:resol1}.
\end{defn}

Let us sketch the proof of the independency from \eqref{eq:resol1}.
Let $0 \to N' \to P' \to M \to 0$ be another exact sequence with $P' \in \cP_R$.
Then, by using the pull-back $L$ of the maps $P \to M$ and $P' \to M$, we obtain a commutative diagram with exact rows and columns
\[
\xymatrix{
	& & N' \ar@{=}[r] \ar@{^(->}[d]
	& N' \ar@{^(->}[d]
	& \\
	0 \ar[r]
	& N \ar[r] \ar@{=}[d]
	& L \ar[r] \ar@{->>}[d]
	& P' \ar[r] \ar@{->>}[d]
	& 0\\
	0 \ar[r]
	& N \ar[r]
	& P \ar[r]
	& M \ar[r]
	& 0.
}
\]
Since $P, P' \in \cP_R$, by Lemma \ref{lem:FittP}(2), we obtain
\[
\Fitt_R(L) = \Fitt_R(P') \Fitt_R(N),
\quad
\Fitt_R(L) = \Fitt_R(P) \Fitt_R(N').
\]
These formulas imply 
\[
\Fitt_R(P)^{-1} \Fitt_R(N)
= \Fitt_R(P')^{-1} \Fitt_R(N'),
\]
as desired.

Finally let us introduce the general $n$-th shift.
See \cite[Theorem 2.6]{Kata_05} for the proof; to remove the noetherian hypothesis, we only have to suitably deal with the finiteness conditions.

\begin{defn}\label{defn:shift_n}
Let $n \geq 0$ be an integer.
Let $M$ be a torsion $R$-module such that there exists an exact sequence
\[
R^{a_n} \to R^{a_{n-1}} \to \dots \to R^{a_0} \to M \to 0
\]
for some integers $a_0, \dots, a_n \geq 0$.
For instance, when $n = 0$ (resp.~$n = 1$), this condition means that $M$ is finitely generated (resp.~finitely presented).
For such an $M$, we define a fractional ideal $\Fitt^{[n]}_R(M)$ as follows.
Let us take an exact sequence of $R$-modules
\begin{equation}\label{eq:resol2}
0 \to N \to P_1 \to \cdots \to P_n \to M \to 0
\end{equation}
with $P_1, \dots, P_n \in \cP_R$.
The existence of such a sequence follows from the assumption on $M$, and moreover then $N$ is finitely generated.
Then we define (using Lemma \ref{lem:FittP}(1))
\[
\Fitt^{[n]}_R(M) = \left( \prod_{i=1}^n \Fitt_R(P_i)^{(-1)^i} \right) \Fitt_R(N).
\]
This is well-defined, that is, independent from the choice of \eqref{eq:resol2}.
\end{defn}

\section*{Acknowledgments}

I am grateful to Daniel Valli\`eres for giving comments on an earlier version of this paper.
This work is supported by JSPS KAKENHI Grant Number 22K13898.

{
\bibliographystyle{abbrv}
\bibliography{biblio}
}

\end{document}